%% file: dilare.tex
\documentclass{article}
\usepackage{amsmath}
\usepackage[utf8]{inputenc}
\usepackage{amssymb}
\usepackage{xparse}
\usepackage{physics}
\usepackage{amsthm}
\usepackage{mathrsfs}
\usepackage{verbatim}
\usepackage{dsfont}
\usepackage{tikz}
\usepackage{qtree}
\usepackage{forest}
\usepackage{mathtools}
\usepackage{cancel}
\usepackage{hyperref}
\usepackage{MnSymbol}
\usepackage{tikzit}
\usepackage{blindtext}
\newtheorem{theorem}{Theorem}
\newtheorem{lemma}{Lemma}
\newtheorem{cor}{Corollary}
\theoremstyle{definition}
\newtheorem{defn}{Definition}
\theoremstyle{remark}
\newtheorem{clm}[theorem]{Claim}
\newtheorem{conj}{Conjecture}

\newcommand{\defref}[1]{\hyperref[#1]{Definition \ref*{#1}}}
\newcommand{\Defref}[1]{\hyperref[#1]{Definition \ref*{#1}}}
\newcommand{\lemref}[1]{\hyperref[#1]{Lemma \ref*{#1}}}
\newcommand{\Lemref}[1]{\hyperref[#1]{Lemma \ref*{#1}}}
\newcommand{\thmref}[1]{\hyperref[#1]{Theorem \ref*{#1}}}
\newcommand{\Thmref}[1]{\hyperref[#1]{Theorem \ref*{#1}}}
\newcommand{\coref}[1]{\hyperref[#1]{Corollary \ref*{#1}}}
\newcommand{\Corref}[1]{\hyperref[#1]{Corollary \ref*{#1}}}

 \tikzset{
          solid node/.style={circle,draw,inner sep=1.2,fill=black},
          green node/.style={circle,draw,inner sep=1.2,fill=green, draw=green},
          red node/.style={circle,draw,inner sep=1.2,fill=red, draw=red},
          blue node/.style={circle,draw,inner sep=1.2,fill=blue, draw=blue},
          yellow node/.style={circle,draw,inner sep=1.2,fill=yellow, draw=yellow},
          cyan node/.style={circle,draw,inner sep=1.2,fill=cyan, draw=cyan},
          hollow node/.style={circle,draw,inner sep=1.2},
          big node/.style={elipse, draw=green!60, fill=green!5},
          big solid node/.style={circle,draw,inner sep=1.2,fill=black, minimum size=0.3cm},
          big green node/.style={circle,draw,inner sep=1.2,fill=green, draw=green, minimum size=0.3cm},
          big red node/.style={circle,draw,inner sep=1.2,fill=red, draw=red, minimum size=0.3cm},
          big blue node/.style={circle,draw,inner sep=1.2,fill=blue, draw=blue, minimum size=0.3cm},
          big yellow node/.style={circle,draw,inner sep=1.2,fill=yellow, draw=yellow, minimum size=0.3cm},
          big cyan node/.style={circle,draw,inner sep=1.2,fill=cyan, draw=cyan, minimum size=0.3cm},
          left label/.style={above left,midway},
          right label/.style={above right,midway}
        }
\newcommand*{\threesim}{%
    \mathrel{
        \vcenter{
            \offinterlineskip
            \hbox{$\sim$}\vskip-.35ex\hbox{$\sim$}\vskip-.35ex\hbox{$\sim$}
            }
        }
    }

\input{node.tikzstyles}
\begin{document}
    \title{Classification of  Label-Regular Directed Trees up to Almost Isomorphism}
    \author{Roman Gorazd\footnote{University of Newcastle}}
    \maketitle
    \begin{abstract}
        This paper outlines a method to determine whether two label-regular directed trees, are isomorphic and when they are almost isomorphic. The approach involves reinterpreting label-regular directed trees as universal covers of rooted graphs. This allows us associate a unique graph with each isomorphism class of a label-regular directed tree. Additionally, by examining the graph monoid we can verify when two unfolding graphs produce almost isomorphic unfolding trees, thereby classifying unfolding trees up to almost isomorphism.
    \end{abstract}

    \section*{Highlights}
    \begin{itemize}
        \item Showing when two unfolding trees of directed graphs are isomorphic via non-edge-collapsing equivalence relations.
        \item Showing when two unfolding trees of directed graphs are almost isomorphic using the graph monoid
        \item Conjecture an explicit classification of unfolding trees of graphs with two vertices. 
    \end{itemize}
    \newpage
    \section{Introduction}
    In the study of totally disconnected locally compact groups, one of the most important examples are groups acting on trees. In the the context of totally disconnected locally compact groups, an often explored family of trees are label regular trees, i.e. trees with a colouring where the neighbourhood of two vertices with the same colour is coloured the same way. When we look at rooted trees oriented away from the root and impose the same condition except just looking at the outgoing neighbourhood, we obtain the class of label-regular directed trees, that in fact contains label-regular trees (with any direction). This change of the definition allows us to easier determine the structure of almost automorphisms on the trees as well as allows us to describe them as unfolding trees of directed graphs.\\
    Finite-depth unfolding trees of graphs are a tool often used while studying neural networks as in \cite{NeuralNet2009}. In this paper we will explore the infinite depth unfolding tree that is a slight expansion of the unfolding trees introduced in \cite{lederle2020topological}. We will draw a connection between the structure of these trees and the graph monoids of the respective graphs. In fact one can show that the word problem in graph monoids is computationally equivalent to the problem of determining when the unfolding trees of two graphs are almost isomorphic. This in turn leads us to estimate the complexity of calculating when two graphs produce almost isomorphic  unfolding trees using the results from \cite{MAYR1982}.\\ 
    Group monoids naturally arise from the study of Leavitt path algebras(introduced in \cite{leavitt1962module}), but can be also very easily be defined via an explicit presentation. It is a popular subject of study \cite{ara2007},\cite{LI2003105} that often look at extensions of the concept \cite{cordeiro2021talented}.\\
    In a follow-up paper I will use the tools that have been introduced in this to explore the Almost Automorphism and the Higman-Thompson groups of these trees and their action on the boundary.\\
    The main goals of this paper is classifying the class of label-regular directed  trees up to isomorphism and up to almost isomorphisms. In order to do this we note that any label-regular directed tree is an unfolding tree of some rooted graph.
    Now to classify the label-regular directed trees up to isomorphic we will look at non-edge collapsing equivalence relations (i.e. relations that induce bijections on the neighbourhoods of equivalent vertices) on graphs and note that when we factor over it we do not change the unfolding structure.
    \begin{theorem}
        Two rooted graphs $G,H$  have isomorphic unfolding trees if and only if there are two non-edge-collapsing equivalence relations $\sim_G,\sim_H$ on $G,H$ resp. s.t. $G/{\sim_G}$ and $H/{\sim_H}$ are rooted-isomorphic.\\
        So each label-regular directed tree is isomorphic to an unfolding tree of a unique (up to isomorphism) graph that has no non-edge-collapsing equivalence relations (non-redundant graphs).
    \end{theorem}
    In order to reduce a graph to the non-redundant graph that produces the same unfolding tree we can apply the first phase of the \textbf{nauty} algorithm as described in \cite{codish2016} and factor over the obtained partition.\\
    This allows us to look only at unfolding trees of non-redundant graphs. When classifying these trees up to almost isomorphism the key observation is to note that two unfolding trees of one graph rooted at different vertices are almost isomorphic to each other if and only if their roots are equal to each other in the graph monoid. This parallel between graph monoids and the almost-structure of unfolding trees is the key insight needed to classify label regular directed trees up to almost isomorphisms.
    In order to take advantage of this insight we note that when working with robustly rooted graphs if the unfolding trees are almost isomorphic the underling graphs are isomorphic (but not necessary rooted-isomorphic).\\
    These two results allow us to classify unfolding trees of robustly rooted graphs. We can expand it to all graphs by seeing that:
    \begin{theorem}
        For any rooted graph $G$ there exists $G_1,\dots,G_n$ robust graphs with $k_1,\dots, k_n\in \mathds{Z}_{>0}$ and functions $\rho_1,\dots,\rho_n$ s.t. the spider product:
        \[\llangle(G_1,k_1,\rho),\dots,(G_n,k_n,\rho_n)\rrangle\] 
        has almost isomorphic tree with $G$
    \end{theorem}
    Here a spider product connects various graphs to one root as shown in figure \ref{fig1}.
    \begin{figure}
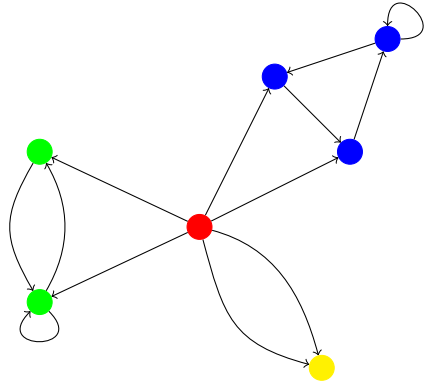
\label{fig1}
    \ctikzfig{spider_product_1}
    \caption{An example of a spider product of 3 graphs (with green, yellow and blue vertices) and a red root}
    \end{figure}
    So we only have to look at spider products of robust graphs. We note first that if they have almost isomorphic unfolding trees we may assume that they have the same component graphs, thanks to the fact that if unfolding trees of robustly rooted graphs are almost isomorphic, the graphs are isomorphic. Additionally we can see that:
    \begin{theorem}
        For any two non redundant spider products of  robust graphs:
        \begin{align*}
            (G_{\rho},S_{\rho}):&=\llangle (G_1,\rho_1.k_1),\dots,(G_n,\rho_n.k_n) \rrangle\\
            (G_{\sigma},S_{\sigma}):&=\llangle (G_1,\sigma_1.l_1),\dots,(G_n,\sigma_n.l_n) \rrangle\\
        \end{align*}
        then $\mathcal{T}(G_{\rho},S_{\rho})$ and $\mathcal{T}(G_{\sigma},S_{\sigma})$ are almost isomorphic if and only if for each $m\in\{1,\dots,n\}$ we have:
        \[\sum_{i=1}^{k_m}\rho_m(i)=\sum_{j=1}^{l_m}\sigma_m(j)\]
        in the graph monoid of $G_m$.
    \end{theorem}
    This allows us to classify all unfolding trees up to almost isomorphism by looking at each spider product of connected robust graphs and checking when the sums in the above theorem are equal.  
    \section{Basic Definitions}
    In this document we will be exploring the structure of label regular directed (l.r.d.) trees that are an expansion of the concept of label regular trees that can be described by directed graphs with multiple edges. This description will allow us to better characterise its almost automorphism group.\\
    In this paper we will view define graphs in a very broad way by saying that a graph $G$ is a $4$-tuple $(V,E,o,t)$ where $V$ and $E$ are sets (of vertices and edges resp.) and $o,t$ are functions from $E$ to $V$, denoting the origin and terminus of an edge respectively. We will call $G$ finite if $V$ and $E$ are finite.\\
    A subgraph $H=(VH,EH,o_H,t_H)$ of $G=(VG,EG,o_G,t_G)$ is a graphs s.t.:
    \begin{align*}
        VH&\subseteq VG\\
        EH&\subseteq EG\\
        o_H&=o_G|_{EH}\\
        t_H&=t_G|_{EH}
    \end{align*}
    The union of two subgraphs $H_1,H_2$ will be the subgraph with vertex set $VH_1\cup VH_2$ and edge set $EH_1\cup EH_2$, we will denote it by $H_1\cup H_2$. The intersection $H_1\cap H_2$ is defined analogously. 
    We will call a subraph $H$ vertex induced if we have:
    \[EH=\{e\in EG\mid o_G(e),t_G(e)\in VH \}\]
    A graph homomorphism $\phi$ between two graphs $G=(V,E,o,t)$ and\\
    $G'=(V',E',o',t')$ consists of two functions $\phi_V:V\to V'$ and $\phi_E:E\to E'$ s.t.:
    \[\forall e\in E\ \begin{cases}
        o'(\phi_E(e))=\phi_V(o(e))\\
        t'(\phi_E(e))=\phi_V(t(e))
    \end{cases}\ \]
    If $\phi_V$ and $\phi_E$ are both injective or surjective , we call $\phi$ injective or surjective, resp., if $\phi_V$ and $\phi_E$ are bijections we call $\phi$ an isomorphism. For two isomorphic graphs $G$ and $H$ we write $G\cong H$.
    For any graph we will define the outgoing and incoming edge-neighbourhood of some vertex $v\in V$ as follows:
    \begin{align*}
        E_o(v):&=\{e\in E\mid o(e)=v\}\\
        E_t(v):&=\{e\in E\mid t(e)=v\} 
    \end{align*}
    The outgoing-degree and incoming-degree of a vertex respectively as follows:
    \begin{align*}
        d_o(v):&=|E_o(v)|\\
        d_t(v):&=|E_t(v)|
    \end{align*}
    From this we also can define the incoming and outgoing neighbourhood:
    \begin{align*}
        N_o(v):&=t[E_o(v)]\\
        N_t(v):&=o[E_t(v)]
    \end{align*}
    If $E_o(v)=\emptyset$ we will call $v$ a sink and if $E_t(v)=\emptyset$, a source.\\
    Now we will define the space of walks on $G$ to be:
    \[ \mathcal{W}(G):=\{e_1\dots e_n\in E^*\mid \forall 1\leq i< n,\ t(e_i)=o(e_{i+1}) \}\cup \{\epsilon_v\mid v\in VG\} \]
    Where $\epsilon_v$ denotes the empty walk on $v$. For any non-empty walk $p=e_1\dots e_n$, define $O(p):=o(e_1)$ and $T(p):=t(e_n)$, additionally the length of $p$ will be $|p|:=n$. For empty walks $\epsilon_v$ we define $O(\epsilon_v):=v=:T(\epsilon_v)$ and $|\epsilon_v|:=0$.\\
    A graph $G$ is \textbf{disconnected} if there are two non-empty subgraphs $H_1,H_2$ s.t. $H_1\cap H_2=\emptyset$ and $H_1\cup H_2=G$. If this is not the case we call the graph $G$ \textbf{connected}.\\  
    A graph will be \textbf{star-connected} if and only if for any two vertices $v_1,v_2\in V$ we have some walks $p_1,p_2\in \mathcal{W}(G)$, with $T(p_1)=v_1$, $T(p_2)=v_2$ and $O(p_1)=O(p_2)$. We will say that such a graph has a \textbf{root} $S\in V$ if we can always choose $p_1,p_2$ s.t.~$O(p_1)=O(p_2)=S$. A rooted graph is a pair $(G,S)$ with $G$ being a graph and $S$ is its root. Any two rooted graphs $(G,S)$ and $(H,R)$ are rooted isomorphic if there is a graph isomorphism $\phi: G\to H$ with $\phi(S)=R$, we will write $(G,S)\cong_R (H,R)$.
    For any vertex $x\in VG$ we define the outgoing graph from $x$: $G_x:=(V_x,E_x,o_x,t_x)$ where:
    \begin{align*}
        V_x:&=\{x\}\cup \{y\in V\mid \exists p\in \mathcal{W}(G),\ O(p)=y\land T(p)=x\}\\
        E_x:&=\{e\in E\mid o(e),t(e)\in V_x\}\\
        o_x:&= o|_{E_x}\\
        t_x:&= t|_{E_x}
    \end{align*}
    note that for any such graph $x$ is a natural root. We will call two vertices $x,y\in V$ o-equivalent (we write $x\sim_o y$) if and only if:
        \[(G_x,x)\cong_R(G_y,y)\]
    We will call a graph \textbf{acyclic} if and only if for any non-empty walk $p\in\mathcal{W}(G)$ we have $O(p)\neq T(p)$.\\
    The graph $G$ is said to be \textbf{without clashes} if for any $e_1,e_2\in E$:
    \[t(e_1)=t(e_2)\implies e_1=e_2\]\\
    So we will define a directed \textbf{tree} to be a star connected acyclic graph without clashes. If the tree has a root we will call it rooted (note such a root is always unique and thus any  isomorphism between trees will be rooted).\\
    If $VG$ and $EG$ are both finite sets we will call the graph $G$ finite. 
    Now that we have defined graphs and trees we can define label-regular directed trees in 3 different ways:
    \begin{defn}\label{def1}
        A rooted tree $\mathcal{T}=(V\mathcal{T},E\mathcal{T},o_{\mathcal{T}},t_{\mathcal{T}})$ is label-regular if we have some finite set $M$ and some function $l:V\mathcal{T}\to M$ s.t.~for any $m\in M$:
        \[ \forall v,w\in V\mathcal{T}\ (l(v)=l(w))\implies |\{ u\in N_o(v)\mid l(u)=m  \}|=|\{ u\in N_o(w)\mid l(u)=m  \}| \]
    \end{defn}
    Note that this definition differs from the definition of an undirected label-regular tree, which is a less general notion.
    \begin{defn}\label{def2}
        For a rooted tree $\mathcal{T}$ we say that it is label-regular if and only if:
        \[|\{[x]_{\sim_o}\mid x\in V\mathcal{T}\}|<\infty\]
    \end{defn}
    For the final definition we define for any rooted graph $(G,S)$ the space of walks starting at $S$: $\mathcal{W}(G,S):=\{w\in \mathcal{W}(G)\mid O(w)=S\}$. We can endow this set with a graph structure by defining the graph $\mathcal{T}(G,S)$ as follows:
    \begin{align*}
        V\mathcal{T}(G,S):&=\mathcal{W}(G,S)\\
        E\mathcal{T}(G,S):&=\{ (w,we)\in \mathcal{W}(G,S)\mid w,we\in\mathcal{W}(G,S), e\in EG \}\\
        \forall (w,we)\in E\mathcal{T}(G,S),\ &o_{\mathcal{T}}( (w,we) ):=w\text{ and } t_{\mathcal{T}}( (w,we) ):=we
    \end{align*}
    This clearly is a rooted tree with root $\varepsilon_S$.
    \begin{defn}\label{def3}
        A rooted tree $(\mathcal{T},\epsilon)$ is label-regular if there exists some finite rooted graph $(G,S)$ s.t.
            \[(\mathcal{T},\epsilon)\cong_R (\mathcal{T}(G,S),\epsilon_S)\]
    \end{defn}
    \begin{theorem}\label{thm0}
    The three above definitions are equivalent
    \end{theorem}
    \begin{proof}
    In order to show the above theorem three definitions are equivalent we will show that \defref{def2} implies \defref{def1}, which implies \defref{def3}, which in turn implies \defref{def2}.\\
    For the first implication we will take some rooted tree $\mathcal{T}$ s.t.~ 
        \[|\{[x]_{\sim_o}\mid x\in V\mathcal{T}\}|<\infty\]
    then when we define $M:=\{[x]_{\sim_o}\mid x\in V\mathcal{T}\}$ and $l:V\mathcal{T}\to M$ by setting $l(x):=[x]_{\sim_o}$ for any vertex $x$. Now in order to show that this labelling satisfies the conditions from \defref{def1} we pick any two vertices $x,y$ with $l(x)=l(y)$ this allows us to fix some rooted isomorphism $\Psi:\mathcal{T}_x\to \mathcal{T}_y$. We note that $\Psi(N_o(x))=N_o(y)$ since $\Psi(x)=y$, giving us a bijection $\Psi|_{N_o(x)}$ between $N_o(x)$ and $N_o(y)$. Furthermore we see that $l(z)=l(\Psi(z))$, since $\Psi|_{\mathcal{T}_{z}}:\mathcal{T}_z\to \mathcal{T}_{\Psi(z)}$ is a rooted isomorphism. This shows us that:
    \[|\{ z\in N_o(x)\mid l(z)=m  \}|=|\{ z\in N_o(y)\mid l(z)=m  \}|\]
    for any $m\in M$.\\
    Now if we there is a labelling $l:V\mathcal{T}\to M$ as in definition \ref*{def1} w.l.o.g. we can assume that $l$ is surjective by shrinking $M$. We can for any $m,n\in M$ define $Y_{m,n}$ to be disjoint sets s.t.
    \[|Y_{m,n}|=|\{ u\in N_o(v)\mid l(u)=n  \}|\]
    where $v\in V\mathcal{T}$ is any vertex s.t.~$l(v)=m$. Now if we set:
    \[Y:=\bigcup_{m,n\in M}Y_{m,n}\] 
    We can define the graph $G=(M,Y,o,t)$ by setting:
    \[\forall m,n\in M,\forall e\in Y_{m,n},\ o(e)=m\text{ and }t(e)=n\]
    Now let $s$ be the root of $\mathcal{T}$ and set $S:=l(s)$. We will construct an isomorphism $\Psi:\mathcal{T}(G,S)\to \mathcal{T}$ inductively over the length of a word $w\in \mathcal{W}(G,S)$ s.t.~$l(\Psi(w))=T(w)$. For the induction beginning we define $\Psi(\varepsilon)=s$. Now if we already defined $\Psi(w)=v$ for some word $w\in \mathcal{W}(G,S)$ then set $m:=O(w)=l(v)$, we will fix a bijection:
    \[\pi_{v}:Y_{m}:=\bigcup_{n\in M}Y_{m,n}\to N_o(v)\]
    s.t.~$\forall u\in Y_{m,n},\ l(\pi_v(u))=n$. Note that we have $\{u\in EG\mid wu\in\mathcal{W}(G,S) \}=Y_m$, so we can define $\Psi(wu)=\pi_v(u)$. Now since we have $l(\Psi(wu))=l(\pi_v(u))=t(u)=T(wu)$, the construction is complete.\\
    $\Psi$ is a homomorphism by the construction and since the $\pi_v$ are bijections it is an isomorphism.\\
    Now to see that \defref{def3} implies \defref{def2} we note that for any $w_1,w_2\in \mathcal{W}(G,S)$ s.t.~$T(w_1)=T(w_2)$ we have for any $u\in \mathcal{W}(G)$:
    \[w_1u\in \mathcal{W}(G,S)\iff O(u)=T(w_1)=T(w_2)\iff w_2u\in \mathcal{W}(G,S)\] 
    so we can see that the function:
    \[\Psi_V:   \begin{cases}
                    V\mathcal{T}(G,S)_{w_1}\to V\mathcal{T}(G,S)_{w_2}\\
                    w_1u\mapsto w_2u
                \end{cases} \]
    combined with the function $\Psi_E:E\mathcal{T}(G,S)_{w_1}\to E\mathcal{T}(G,S)_{w_2}$ defined by\\
    $\Psi_E( (w,we) )=(\Psi_V(w),\Psi_V(we))$ gives us a tree isomorphism $\Psi:\mathcal{T}(G,S)_{w_1}\to \mathcal{T}(G,S)_{w_2}$. So any words with the the same terminus are o-equivalent and since there are only finitely many options for the terminus ($G$ is finite), we thus can only have finitely many equivalence classes of vertices.
    \end{proof}
    This allows us to look at label-regular trees as unfolding trees of finite rooted graphs, so instead of looking at these trees we can look at finite rooted graphs. The exploration of the how the properties of a graph connect to the properties of its unfolding tree, allows us to classify these trees up to isomorphism and almost isomorphism. 
\section{Classification of Label-Regular Directed Trees up to Isomorphism}
    Now \defref{def3} gives us a finite way of describing a (often infinite) label-regular trees via rooted graphs. Now we can figure out when two rooted graphs represent isomorphic label-regular trees.
    For this we will first define an equivalence relation on a graph.
    \begin{defn}\label{def4}
        For any graph $G=(V,E,o,t)$ an equivalence relation $\sim$ on $G$ consists of two equivalence relations $\sim_V$ and $\sim_E$ on $V$ and $E$ respectively s.t:
        \[ \forall e_1,e_2\in EG (e_1\sim_E e_2)\implies\big((o(e_1)\sim_V o(e_2))\land(t(e_1)\sim_V t(e_2))\big) \]
        Such an equivalence relation is trivial if both $\sim_V$ and $\sim_E$ are.\\
        The quotient graph $G/{\sim}:=(V/{\sim_V},E/{\sim_E},o_{\sim},t_{\sim})$ is defined by setting: 
        \begin{align*}
            V/{\sim}:&=\{[v]_{\sim_V}\mid v\in V\}\\
            E/{\sim}:&=\{[e]_{\sim_E}\mid e\in E\}\\
            \forall e\in E,\ o_{\sim}([e]_{\sim_E}):&=[o(e)]_{\sim_V}\\
            \forall e\in E,\ t_{\sim}([e]_{\sim_E}):&=[t(e)]_{\sim_V}
        \end{align*}
        We will call a graph relation $\sim$ \textbf{non edge collapsing} if:
        \[\forall v\sim v'\in VG,\forall e\in EG\ (o(e)=v)\implies (\exists! e'\in EG\ (o(e')=v')\land (e\sim e') )\]
    \end{defn}
    We note first that any surjective graph homomorphism that is injective on each outgoing neighbourhood to a non-edge-collapsing equivalence relation (and vice-versa):
    \begin{lemma}\label{lem-1}
        Take any two graphs $G,H$ and any homomorphism:
        \[\phi:G\to H\]
        s.t.~$\phi$ is surjective and for any $v\in VG$: $\phi|_{E_o(v)}$ is  a bijection from $E_o(v)$ to $E_o(\phi(v))$.
         Then there exists a non edge collapsing equivalence relation $\sim$ on $G$ and an isomorphism:
        \[\psi:G/{\sim}\to H\]
        s.t.~$\forall x\in VG\cup EG,\ \psi([x]_{\sim})=\phi(x) $
    \end{lemma}
    We will call any homomorphism $\phi$ as in the lemma \textbf{non edge collapsing}.
    \begin{proof}
        We will define $\sim$ as follows:
        \[\forall x,y\in  VG\cup EG,\ x\sim y \iff \phi(x)=\phi(y)\]
        this is a graph-equivalence relation since $\phi$ is a graph homomorphism. In order to show that $\sim$ is non edge collapsing we will take $v_1\sim v_2$ and some $e_1\in E_o(v_1)$. Now let $e:=\phi(e_1)$ by definition we have $o(e)=\phi(v_1)=\phi(v_2)$ and since $\phi|_{E_o(v_2)}$ is a bijection we have some unique $e_2\in E_o(v_2)$ s.t.~$\phi(e_2)=e=\phi(e_1)$ i.e. $e_2\sim e_1$. This gives us that $\sim$ is non edge collapsing.\\
        Now we can define $\psi:G/{\sim}\to H$ by setting $\forall x\in VG\cup EG,\ \psi([x]_{\sim})=\phi(x) $. $\psi$ is well defined since for any $x\sim y\in VG\cup EG$ we have $\phi(x)=\phi(y)$. $\psi$ is a graph homomorphism since:
        \[\forall e\in EG,\ o(\psi([e]_{\sim}))=o(\phi(e))=\phi(o(e))=\psi([o(e)]_{\sim})=\psi(o([e]_{\sim}))\]
        and analogously:
        \[\forall e\in EG,\ t(\psi([e]_{\sim}))=t(\phi(e))=\phi(t(e))=\psi([t(e)]_{\sim})=\psi(t([e]_{\sim}))\]
        $\psi$ is surjective since $\phi$ is. Finally $\psi$ is injective since:
        \[\forall x,y\in VG\cup EG,\ \psi([x]_{\sim})=\psi([y]_{\sim})\implies \phi(x)=\phi(y)\implies [x]_{\sim}=[y]_{\sim}\]
        Thus $\psi$ is the graph isomorphism we are looking for.
    \end{proof}
    The non-edge collapsing condition means that if two vertices are $o$-equivalent in the graph then they are $o$-equivalent in the quotient graph. This observation leads us to show that any graph has a non-edge collapsing equivalence relation that has $\sim_o$ as its vertex component. Furthermore we can note that the vertex component of the non-edge-collapsing equiv. relation determines the structure of the graphs. Combining these properties allows us to find for each graph a unique minimal factor over a non-edge-collapsing equivalence relation.
    \begin{lemma}\label{lem0}
        Let $G$ be a graph then we have:
        \begin{itemize}
            \item There exists a non edge collapsing equivalence relation whose vertex component is $\sim_o$
            \item For any two non edge collapsing equivalence relations $\sim^1$ and $\sim^2$ with $\sim^1_V=\sim^2_V$  we have a isomorphism:
            \[\phi: G/{\sim^1}\to G/{\sim^2}\]
            with $\forall v\in V \phi([v]_{\sim^1})=[v]_{\sim^2}$
            \item For any non edge collapsing equivalence relation $\sim$ on $G$ and any non edge collapsing equivalence relation $\approx$ on $G/{\sim}$  we have some non edge collapsing equivalence relation $\threesim$ on $G$ s.t.
            \[\forall x,y\in VG\cup EG,\ x\sim y\implies x\threesim y\]
            and we have:
            \[G/{\threesim}\cong(G/{\sim})/\approx\]
            \item For any two non edge-collapsing equivalence relations $\sim^1,\sim^2$ there is another non edge-collapsing relation $\sim$ s.t.
            \[ \forall v,w\in VG,\ v\sim^1_Vw \lor v\sim^2_Vw\implies v\sim_V w \]
         \end{itemize}
    \end{lemma} 
    \begin{proof}
        For the first point we will fix for any two vertices $v\sim_o w$ and a rooted isomorphism $\phi_{v,w}:G_v\to G_w$ we can choose these isomorphisms in such a manner that for any $v\sim_o w\sim_o u$ we have:
        \begin{align*}
            \phi_{v,v}&=\text{id.}\\
            \phi_{v,w}&=\phi_{w,v}^{-1}\\
            \phi_{v,u}&=\phi_{v,w}\circ \phi_{w,u}
        \end{align*}
        Now we can see that if we define $e\sim_e f\iff (o(e)\sim_o o(f))\land \phi_{o(e),o(f)}(e)=f$, this will define an equivalence relation. And furthermore if we have $e\sim_e f$ then by definition we have $o(e)\sim_o o(f)$ and furthermore since $\phi_{o(e),o(f)}(e)=f$ we must have $\phi_{o(e),o(f)}(t(e))=t(f)$ and thus the function $\phi_{o(e),o(f)}|_{G_{t(e)}}$ gives us a rooted isomorphism between $G_{t(e)}$ and $G_{t(f)}$ and thus $t(e)\sim_o t(f)$.\\
        This shows that $(\sim_o,\sim_e)$ is a graph equivalence relation. Now to show that it is non edge collapsing we will look at any edge $e$, with $o(e)=v$ and some vertex $v'\sim_o v$ then we see that $\phi_{v,v'}(e)$ is the unique edge that is equivalent to $e$ and has origin in $v'$. Thus the equivalence relation is non edge collapsing. This gives us the first point.\\
        For the second point we will take two non edge collapsing equivalence relations $\sim^1$ and $\sim^2$ with $\sim_V:=\sim^1_V=\sim^2_V$. Now let the set $V_0\subset V$ be a set of distinct representatives of each equivalence class of $\sim_V$. Now we can define:
        \[E_0:=\bigcup_{v_0\in V_0}E_o(v_0)\]
        we will show that $E_0$ is a set of representatives for each equivalence class of $\sim^1_E$ and $\sim^2_E$. For this we will fix some $i\in \{1,2\}$. Now say we have $e_0\sim^i_E f_0$ for some $e_0,f_0\in E_0$ then we must have $o(e_0)\sim_V o(f_0)$, but since we have by definition $o(e_0),o(f_0)\in V_0$, it must be that $o(e_0)=o(f_0)$. Since the equivalence relation $\sim_e^i$ is non edge collapsing we must have $e_0=f_0$.\\
       For any $e\in E$ we note that we have some $v_0\in V_0$ s.t.~$o(e)\sim_V v_0$. Since $\sim^i$ is non edge collapsing, there must be some $e_0\in E_o(v_0)\subseteq E_0$ s.t.~$e\sim^i_E e_0$.\\
        This shows that $E_0$ is a set of representatives each equivalence class of $\sim^1_E$ and of $\sim^2_E$ (as $i\in\{1,2\}$ was  arbitrary).\\
        Using this we can construct an isomorphism $\Psi:G/{\sim^1}\to G/{\sim^2}$ by setting:
        \begin{align*}
        \forall v\in V,\ \Psi_V([v]_{\sim_V}):&=[v]_{\sim_V}\\
        \forall e_0\in E_0\ \Psi_E([e_0]_{\sim^1_E}):&=[e_0]_{\sim^2_E}
        \end{align*}
        It is clear that $\Psi_V$ is a bijection. The fact that $\Psi_E$ is well defined and also a bijection, follows from $E_0$ being a representative set for both $\sim^1_E$ and $\sim^2_E$.\\
        To show that it is an isomorphism we just look at some $[e]_{\sim^1_E}\in E/{\sim^1_E}$ and pick the unique $e_0\in E_0$ s.t.~$[e]_{\sim^1_E}=[e_0]_{\sim^1_E}$ we then have:
        \begin{align*}
            o(\Psi_E([e]_{\sim^1_E}))&=o(\Psi_E([e_0]_{\sim^1_E}))=o([e_0]_{\sim^2_E})=\Psi_V(o([e_0]_{\sim^2_E}))=\Psi_V(o([e]_{\sim^2_E}))\\
            t(\Psi_E([e]_{\sim^1_E}))&=t(\Psi_E([e_0]_{\sim^1_E}))=t([e_0]_{\sim^2_E})=\Psi_V(t([e_0]_{\sim^2_E}))=\Psi_V(t([e]_{\sim^2_E}))\\
        \end{align*}
        This shows that $\Psi$ is in fact an isomorphism.\\
        For the third point we define $\threesim$ as follows: 
        \[\forall x,y\in VG\cup EG,\ x\threesim y\iff [x]_{\sim}\approx [y]_{\sim}\]
        It is immediate that $\threesim$ is a graph equivalence relation, as $\sim$ and $\approx$ are. To show that it is non edge collapsing take some $v\threesim v'\in VG$ and some $e\in E_o(v)$, since $\approx$ is non edge collapsing we have some unique $[e']_{\sim}\in E_o([v']_{\sim})$ with $[e']_{\sim}\approx[e]_{\sim}$.\\
        Since $\sim$ is also non edge collapsing we may assume that $e'$ is the unique representative of $[v']_{\sim}$ with $o(e')=v'$ then by definition this is the unique $e'\in E_o(v')$ with $e'\threesim e$.\\
        The fact $G/{\threesim}\cong(G/{\sim})/\approx$ follows by a standard result.\\
        For the last point we will define the equivalence relation $\sim$ on the vertices $G$ by setting:
        \[ \forall v,w\in V\ v\sim_V w \iff v\sim^1_V w \lor v\sim^2_V w\]
        To define it on the edges we take some $e,f\in EG$ and consider the cases:
        \begin{itemize}
            \item If $o(e)\sim^1 o(f)$ then we define:
            \[e\sim f\iff e\sim^1 f\]
            \item If we have $o(e)\sim^2 o(f)$, but $o(e)\not\sim^1 o(f)$ we define:
            \[e\sim f\iff e\sim^2 f\]
            \item If we have both $o(e)\not\sim^1 o(f)$ and $o(e)\not\sim^2 o(f)$ we also have $e\not\sim f$.
        \end{itemize}
        To see that this is a graph equivalence relations we note that if we have $e\sim f$ we have $e\sim^if$ for some $i=1,2$ and thus $o(e)\sim^i o(f)$ and $t(e)\sim^i t(f)$ and thus $o(e)\sim o(f)$ and $t(e)\sim t(f)$.\\
        To show that the relation is non edge-collapsing we take some $v\sim v'\in VG$ and some $e\in E_o(v)$. Now if we have $v\sim^1 v'$ we must have, since $\sim^1$ is non edge collapsing, a unique $e'\in E_o(v')$ s.t.~$e\sim^1 e'$. By definition this is also the unique $e'\in E_o(v')$ s.t.~$e\sim e'$.\\
        If we have $v\not\sim^1 v'$, then we must have $v\sim^2 v'$ and since $\sim^2$ is also non edge collapsing we have a unique $e'\in E_o(v')$ s.t.~$e\sim^2 e'$. By definition this is also the unique $e'\in E_o(v')$ s.t.~$e\sim e'$.
    \end{proof}
    We note that if we sort non edge collapsing equivalence relations on $G$ by setting $\sim^1\prec\sim^2$ if and only if:
    \[\forall x,y\in VG\cup EG,\ x\sim^1y\implies x\sim^2y\]
    We note that if we have some $\prec$-maximal non edge collapsing equivalence relations $\sim^1$ and $\sim^2$ we must have by the last point of the above lemma we must have:
    \[\sim^1_V=\sim^2_V\]
    We will be calling a graph that has no non-trivial non edge collapsing equivalence relation $\textbf{non-redundant}$. For any Graph $G$ and any non-edge collapsing equivalence relation $\sim$ on $G$ that is $\prec$-maximal, the graph $G/{\sim}$ is non-redundant by the 3rd point of \lemref{lem0}
    From combining the third and fourth point of the above lemma we can see that all non redundant graph of such a form $G/{\sim}$ will be isomorphic to each other. So we get a unique non-redundant graph that is a factor of $G$.\\
    Now we can show that for any non edge collapsing equivalence relation on a graph the quotient graph will produce the same unfolding tree.
    \begin{lemma}\label{lem1}
        For any non-edge collapsing equivalence relation $\sim$ on a rooted graph $G$ with root $S$, we have:
        \[\mathcal{T}(G,S)\cong_R \mathcal{T}(G/{\sim},[S]_{\sim})\]
    \end{lemma}
    \begin{proof}
        In this proof we will omit the subscripts from $\sim_V$ and $\sim_E$ for notational convenience.\\
        Define the function as follows:
        \[\Psi: \begin{cases}
                    \mathcal{T}(G,S)\to\mathcal{T}(G/{\sim},[S]_{\sim})\\
                    e_1e_2\dots e_n\mapsto [e_1]_{\sim}[e_2]_{\sim}\dots [e_n]_{\sim} 
                \end{cases}\]
        This function is well defined since for any $e_1e_2\dots e_n\in \mathcal{T}(G,S)$ we have $[S]_{\sim}=[o(e_1)]_{\sim}$ and for any $i<n$ we have $[o(e_{i+1})]_{\sim}=[t(e_i)]_{\sim}$. It is clear that $\Psi$ is a tree homomorphism since $\forall e\in EG,ve\in \mathcal{W}(G,S)\Psi(ve)=\Psi(v)[e]_{\sim}$.\\
        To show injectivity assume there are two $e_1e_2\dots e_n\neq d_1d_2\dots d_m\in \mathcal{W}(G,S)$ with $\Psi(e_1e_2\dots e_n)=\Psi(d_1d_2\dots d_m)$. Firstly since $\Psi$ preserves the word length we must have $n=m$. Secondly since the words are different let $k$ be the minimal index s.t.~$e_k\neq d_k$. If $k=1$ then we must have $o(e_k)=S=o(d_k)$, otherwise we have $o(e_k)=t(e_{k-1})=t(d_{k-1})=o(d_k)$ since $e_{k-1}=d_{k-1}$. So in each case $o(e_k)=o(d_k)$, but since $e_k\neq d_k$ and as $\sim$ is non-edge-collapsing we must have $e_k\not\sim d_k$ and thus $[e_k]_{\sim}\neq[d_k]_{\sim}$. However this gives us:
        \[\Psi(e_1e_2\dots e_n)= [e_1]_{\sim}[e_2]_{\sim}\dots [e_n]_{\sim} \neq [d_1]_{\sim}[d_2]_{\sim}\dots [d_n]_{\sim}=\Psi(d_1d_2\dots d_n) \]
        contradicting our assumption. This gives us injectivity.\\
        For the surjectivity we take some $[e_1]_{\sim}[e_2]_{\sim}\dots [e_n]_{\sim}\in \mathcal{W}(G/{\sim},[S]_{\sim})$ and recursively pick representatives $e'_1,e'_2,\dots e'_n$ of $[e_1]_{\sim},[e_2]_{\sim},\dots,[e_n]_{\sim}$ respectively, that form a walk in $G$ starting in $S$, as follows.
        Since we have $o(e_1)\sim S$, by our assumption there is some $e'_1$ with $e_1\sim e'_1$ and $o(e'_1)=S$. Now recursively when we have already chosen $e'_1,\dots,e'_k$ for some $k<n$ then since $e'_k\sim e_k$ we have $t(e'_k)\sim t(e_k)\sim o(e_{k+1})$, so again we can find an unique edge $e'_{k+1}$ s.t.~$o(e'_{k+1})=t(e'_k)$ and $e'_{k+1}\sim e_{k+1}$. This way we find $e'_1,e'_2,\dots, e'_n$ s.t.~$e'_1e'_2\dots e'_n\in\mathcal{W}(G,S)$ and $e_1\sim e'_1,\dots, e_n\sim e'_n$, this gives us:
        \[\Psi(e'_1e'_2\dots e'_n)=[e'_1]_{\sim}[e'_2]_{\sim}\dots [e'_n]_{\sim}=[e_1]_{\sim}[e_2]_{\sim}\dots [e_n]_{\sim} \]
        showing surjectivity.\\
        So we have shown that $\Psi$ is a tree isomorphism. Clearly we also have $\Psi(\varepsilon_S)=\varepsilon_{[S]_{\sim}}$ thus making it a rooted isomorphism.
    \end{proof}
    We can show that the above lemma describes all the cases of two graphs that produce the same tree for this we will define two equivalence relations on the class $\Gamma$ of pairs $(G,S)$, where $G$ is a rooted graph and $S$ is a root of $G$. 
    \begin{defn}\label{def5}
        Let $\simeq$ be the smallest equivalence relation on $\Gamma$ s.t.~any root isomorphic rooted graphs are $\simeq$ equivalent and for any $(G,S)\in \Gamma$ and any non edge collapsing equivalence relation on $G$: $\sim$, we have:
        \[(G,S)\simeq(G/{\sim},[S]_{\sim})\]
    \end{defn}
    This equivalence relation exactly describes when two graphs produce the same trees:
    \begin{theorem}\label{thm1}
        For any $ (G,S),(H,R)\in \Gamma$ we have:
        \[ \mathcal{T}(G,S)\cong_R \mathcal{T}(H,R)\]
        if and only if:
        \[(G,S)\simeq (H,R)\]
    \end{theorem}
    \begin{proof}
        The ``if'' part of the implication follows directly from \lemref{lem1}.\\
        For the converse, we first need:
        \begin{clm}\label{clm0}
            For each rooted graph $(G,S)$ there is a non-edge-collapsing equivalence relation $\sim$ on $\mathcal{T}(G,S)$ s.t.
            \[ G\cong_R \mathcal{T}(G,S)/{\sim} \]
        \end{clm}
        \begin{proof}
            To prove this we will define a homomorphism $\Phi:\mathcal{T}(G,S) \to G$ as follows:
            \begin{align*}
                \forall p\in\mathcal{W}(G,S),\ & \Phi(p):=T(p)\\
                \forall p\in\mathcal{W}(G,S),\forall e\in E_o(T(p)),\ & \Phi(p,pe):= e
            \end{align*}
            note that this is a homomorphism, that is surjective since $S$ is a root. Also the function $\Phi|_{E_o(v)}:E_o(v)\to E_o(\Phi(v))$ is a bijection by the construction of the unfolding tree. So by \lemref{lem-1} there is some non edge collapsing equivalence relation $\sim$ s.t.:
            \[ G\cong_R \mathcal{T}(G,S)/{\sim} \]
        \end{proof}
        Now to prove the converse implication, using the above claim we first take some tree $\mathcal{T}$ s.t.
        \[\mathcal{T}(G,S)\cong\mathcal{T}\cong\mathcal{T}(H,R)\]
        By the above claim we then have two non edge collapsing equivalence relations $\sim^1$ and $\sim^2$, s.t.:
        \[G\cong_R\mathcal{T}/{\sim^1}\qquad H\cong_R\mathcal{T}/{\sim^2}\]
        so:
        \[G\simeq\mathcal{T}\simeq H\]
        and thus:
        \[G\simeq H\]

    \end{proof}
     In order to better describe which graphs produce the same unfolding tree we will have to take a closer look at $\simeq$:
    \begin{lemma}\label{lem1.5}
        For any two rooted graphs $(G,S),(H,R)\in \Gamma$ we have:
        \[(G,S)\simeq (H,R)\]
        if and only if there are some non-edge collapsing equivalence relations $\sim_G$, $\sim_H$ on $G$ and $H$ resp. s.t.:
         \[(G/{\sim_G},[S]_{\sim_G})\cong_R (H/{\sim_H},[R]_{\sim_H})\]
    \end{lemma}
    \begin{proof}   
        For the "only if" implication we note that for any two $\sim_G$, $\sim_H$ as in the lemma we have by \defref{def5}:
        \[(G,S)\simeq(G/{\sim_G},[S]_{\sim_G})\simeq (H/{\sim_H},[R]_{\sim_H})\simeq(H,R)\]
        and thus by transitivity we must have:
        \[(G,S)\simeq (H,R)\]
        For the converse implication we will use the proof of \thmref{thm1}. From there we can see that if we have  some tree $\mathcal{T}$ and two non edge collapsing equivalence relations $\sim_1$ and  $\sim_2$ on it s.t.
            \[G\cong_R\mathcal{T}/{\sim^1}\qquad H\cong_R\mathcal{T}/{\sim^2}\]
        Now take $\approx^1$ and $\approx^2$ to be some maximal non-edge collapsing relations on $\mathcal{T}/{\sim^1}$ and $\mathcal{T}/{\sim^2}$. Depending on the context we will also reinterpret them as non-edge collapsing equivalence relations on $G$ and $H$ respectively.\\
        Now by the 3rd point of \lemref{lem0} we have some other equivalence relations $\threesim^1$ and $\threesim^2$ s.t.
            \[ (\mathcal{T}/{\sim^1})/{\approx^1}\cong \mathcal{T}/{\threesim^1} \quad (\mathcal{T}/{\sim^2})/{\approx^2}\cong \mathcal{T}/{\threesim^2}\]
        Now since $\approx^1$ and $\approx^2$ are maximal, so is $\threesim^1$ and $\threesim^2$, due to the proof of the third point of \lemref{lem0}. Thus we must have by the second point of \lemref{lem0} (using the obvious root):
            \[G/{\approx^1}\cong_R (\mathcal{T}/{\sim^1})/{\approx^1}\cong_R \mathcal{T}/{\threesim^1}\cong_R \mathcal{T}/{\threesim^2}\cong_R (\mathcal{T}/{\sim^2})/{\approx^2}\cong_R H/{\approx^2} \]
        Showing the other direction of the lemma.     
    \end{proof}
    \thmref{thm1} immediately gives  us:
    \begin{cor}\label{cor1}
        For any two non-redundant rooted graphs $ (G,S),(H,R)\in \Gamma$ we have:
        \[ \mathcal{T}(G,S)\cong_R \mathcal{T}(H,R)\]
        if and only if:
        \[(G,S)\cong_R (H,R)\] 
    \end{cor}
    \begin{proof}
        From \thmref{thm1} we see that for any such $ (G,S),(H,R)$ we must have: 
        \[(G,S)\simeq (H,R)\]
        i.e. by \defref{def5} we must have non-edge collapsing equivalence relations $\sim_G$, $\sim_H$ on $G$ and $H$ resp. s.t.:
         \[(G/{\sim_G},[S]_{\sim_G})\cong_R (H/{\sim_H},[R]_{\sim_H})\]
         and since $G$ and $H$ are non-redundant these relations must be trivial and thus:
         \[(G,S)\cong_R (H,R)\]
    \end{proof}
    So by looking at two graph we can see when they both produce the same tree. Combining this with \lemref{lem0} we can see that every rooted isomorphism class of rooted label regular trees has a unique non-redundant rooted graph associated to it. This allows us to classify each label regular tree by associating it with the unique non-redundant rooted graph that it is an unfolding tree of. 
    \section{Classification of Label-Regular Directed trees up to almost isomorphism}
    Now we can define when two trees are isomorphic up to a finite subtree i.e. are almost isomorphic. In order to do this we first define what it means to take away a subgraph. 
    \begin{defn}\label{def6}
        For any graph $G$ and any subgraph $H$ of $G$ we can define the graph $G\setminus H$ by setting:
        \begin{align*}
            E(G\setminus H):&=EG\setminus EH\\
            V(G\setminus H):&=VG\setminus \{v\in VG\mid E_o(v),E_t(v)\neq \emptyset,\ E_o(v)\cup E_t(v)\subseteq EH\}\\
            o_{G\setminus H}:&= o|_{E(G\setminus H)}\\
            t_{G\setminus H}:&= t|_{E(G\setminus H)}
        \end{align*}
    \end{defn}
    This allows us to define almost isomorphisms:
    \begin{defn}\label{def7}
        For any two trees $\mathcal{T}_1,\mathcal{T}_2$  we say that they are almost isomorphic if there exists finite subtrees $T_1,T_2$  s.t.~there exists an isomorphism:
        \[\Psi:\mathcal{T}_1\setminus T_1\to \mathcal{T}_2\setminus T_2\]
        We will write $\mathcal{T}_1\cong_A\mathcal{T}_1$ if that is the case. 
    \end{defn}
    Note that by expanding the subtrees $T_1$ and $T_2$ to ensure they contain a potential root. We also can ensure that for any vertex $v$ in $T_i$, we either have $N_o(v)\subseteq T_i$ or $N_o(v)\cap T_i=\emptyset$. We can describe a rooted tree without such a subtree  in the case of the tree being $\mathcal{T}(G,S)$, in a similar way as has been done in \cite{scott84} for regular trees:
    \begin{defn}\label{def8}
        For any rooted graph $G$ and any root $S$, a vertex induced subgraph $\mathcal{S}\leq\mathcal{T}(G,S)$ is:
        \begin{itemize}
            \item A \textbf{subspace} if $\forall w\in\mathcal{W}(G,S),\forall u\in \mathcal{W}(G,T(w))\ w\in \mathcal{S}\implies wu\in \mathcal{S}$
            \item \textbf{Inescapable} if $\forall w\in \mathcal{W}(G,S)\exists u\in \mathcal{W}(G,T(w)):\ wu\in \mathcal{S}$
            \item \textbf{Cofinite} if there is a finite set $F\subseteq V\mathcal{T}(G,S)$ s.t.
                \[\mathcal{S}=\bigcup_{x\in F}\mathcal{T}(G,S)_x\] 
        \end{itemize}
    \end{defn}
    Note that if we have $|\mathcal{W}(G,S)\setminus\mathcal{S}|<\infty$ the subspace $\mathcal{S}$ is cofinite. If $\mathcal{S}$ is inescapable the converse implication is also true.\\ 
    We can show that when talking about almost isomorphisms we can look at inescapable, cofinite subspaces instead of taking away any finite tree.
    \begin{lemma}\label{lem2}
        For any two finite rooted graphs $(G,S)$,$(G',S')$ if the trees $\mathcal{T}(G,S)$ and $\mathcal{T}(G',S')$ are almost isomorphic, then there are some inescapable, cofinite subspaces $\mathcal{S},\mathcal{S}'$ of $\mathcal{W}(G,S)$ and $\mathcal{W}(G',S')$ resp. that are isomorphic.
    \end{lemma}
    \begin{proof}
        We will denote the length of a word $w$ by $|w|$. Since $\mathcal{T}(G,S)$ and $\mathcal{T}(G',S')$ are almost isomorphic we have some finite subtrees $T\leq \mathcal{T}(G,S)$, $T'\leq \mathcal{T}(G',S')$ and an isomorphism:
        \[\phi: \mathcal{T}(G,S)\setminus T\to \mathcal{T}(G',S')\setminus T'\]
        Now set $m':=\max\{|w|\mid w\in VT'\}$ and define the vertex induced subgraph $T'_0$ by setting:
        \[VT'_0:=\{w\in\mathcal{T}(G,S)\setminus T'\mid |w|\leq m'\} \]
        Now we can set $m:=\max\{|w|\mid w\in VT\cup \phi^{-1}[VT'_0]\}$ and define $\mathcal{S}$ to be the vertex induced subgraph of $\mathcal{T}(G,S)$ with:
        \[ V\mathcal{S}=\{w\in\mathcal{T}(G,S)\mid |w|>m\}\cup\{u\in\mathcal{T}(G,S)\mid N_o(u)=\emptyset\}\]
        We first note that $\mathcal{S}$ is indeed an inescapable subspace since we can expand any word to a word with length $>m$, or a leaf (vertex with empty outgoing neighbourhood). To see that it is cofinite we simply note that:
        \[\mathcal{W}(G,S)\setminus V\mathcal{S}\subseteq \{w\in\mathcal{T}(G,S)\mid |w|\leq m\}\]
        and since the left hand set is finite the subspace is cofinite.\\
        We now observe that since by definition we must have for any leaf $u$: $u\in V(\mathcal{T}(G,S)\setminus T)$ and we must have $\{w\in\mathcal{T}(G,S)\mid |w|>m\}\subseteq V(\mathcal{T}(G,S)\setminus T)$ as $V(\mathcal{T}(G,S))\setminus VT\subseteq V(\mathcal{T}(G,S)\setminus T)$. This gives us $\mathcal{S}\subseteq \mathcal{T}(G,S)\setminus T$, so when we define $\mathcal{S}':=\phi[\mathcal{S}]$, we can define the isomorphism:
        \[\psi: \mathcal{S}\to \mathcal{S}' \]
        by defining $\psi:=\phi|_{\mathcal{S}}$. Now we only have to show that $\mathcal{S}'$ is a cofinite inescapable subspace.\\
        To show that it is a subspace we will show that for any $w'\in \mathcal{S}'$ and any\\
        $e'\in E_o(T(w'))$, we have $w'e'\in \mathcal{S}'$; this will show that $\mathcal{S}'$ is a subspace by induction. To show this proposition note that since  $w'\in \mathcal{S}'$ and $E_o(w)\neq\emptyset$ we must have some $w\in \mathcal{S}$ s.t.~$w'=\phi(w)$ and $E_o(w)\neq \emptyset$, so $|w|>m$ giving us $w'=\phi(w)\notin T'_0$ and thus $|w'|>m'$. Thus we must also have $|w'e'|>m'$ and thus $w'e'\in \mathcal{T}(G',S')\setminus T'$. So $w'$ and $w'e'$ are in the domain of $\phi$. Setting $w:=\phi^{-1}(w')\in\mathcal{S}$ and since $\phi^{-1}(w'e')=we$ for some $e\in EG$, we must have $\phi^{-1}(w'e')\in \mathcal{S}$ and thus $w'e'\in \mathcal{S}'$.\\
        To show that the subspace is inescapable we take some $w'\in \mathcal{W}(G',S')$, then we must have some $u'\in \mathcal{W}(G')$ s.t.~$w'u'\in \mathcal{T}(G',S')\setminus T' $ and thus we can set $w=\phi^{-1}(w'u')$ then since $\mathcal{S}$ is inescapable there must be some $u\in\mathcal{W}(G)$ s.t.~$wu\in \mathcal{S}$. Furthermore since $\phi$ is a homomorphism there must be some $u''$ s.t.~$\phi(wu)=w'u'u''\in \mathcal{S}'$.\\
        Finally for cofiniteness we will note that for any $w'\notin V\mathcal{S}'$ we must either have $w'\notin V(\mathcal{T}(G',S')\setminus T')$ or $\phi^{-1}(w')\notin V\mathcal{S}$. Since both of these options are only true for finitely many $w'$, $\mathcal{S}'$ is cofinite. 
    \end{proof}
    So to better describe the label regular trees that are almost isomorphic we will have to take a look at subspaces. For this we will call for any two words $v,w\in \mathcal{W}(G,S)$ $v$ a prefix of $w$ if there is some other word $u$ in $EG^*$, s.t.~$w=vu$, we will write $v\leq_p w$. Note that this order makes $\mathcal{W}(G,S)$ into a semi-lattice. Now we can introduce the basis of a subspace:
    \begin{lemma}\label{lem3}
         For any rooted graph $G$ and some root $S$, for any subspace $\mathcal{S}\leq \mathcal{T}(G,S)$ there exists a unique set $B\subseteq \mathcal{S}$, s.t.~no two distinct elements of $B$ are prefixes of each other and any word in $\mathcal{S}$ has a prefix in $B$. It is the case that:
         \begin{itemize}
              \item The subspace $\mathcal{S}$ is inescapable iff for each $v\in \mathcal{W}(G,S)$ we have some $b\in B$ s.t.~$v\leq_p b$ or $b\leq_p v$.
              \item The subspace $\mathcal{S}$ is cofinite iff $|B|< \infty$.
          \end{itemize}
          We will call the set $B$ the basis of $\mathcal{S}$.
          Conversely for any set $B$ s.t.~no two distinct elements of $B$ are prefixes of each other there exists a unique subspace $\mathcal{S}_B$, s.t.~$B$ is the basis of $\mathcal{S}$.
    \end{lemma} 
    \begin{proof}
        The existence and uniqueness of a basis follows directly from the fact that the partial order forms a semi-lattice.\\
        To show the first point we first assume that $\mathcal{S}$ is inescapable then for any $w\in\mathcal{W}(G,S)$ we have some $v\in\mathcal{S}$ s.t.~$w\leq_p v$. Now for $v$ we have in turn some $b_v\in B$ with $b_v\leq_p v$. So $w$ and $b_v$ are both prefixes of $v$ thus they must be $\leq_p$ comparable.\\
        For the converse let $B$ be the basis of $\mathcal{S}$ s.t.~any word in $\mathcal{W}(G,S)$ is comparable to some word in $B$. Then for any word $v$ we have some comparable $b\in B$, if $b\leq_p v$ then $v\in \mathcal{S}$ by definition of a subspace, on the other hand if $v\leq_p b $ then $v$ is a prefix of some word in $\mathcal{S}$.\\
        For the second point we note that for any subspace $\mathcal{S}$ with basis $B$ we have:
        \[\mathcal{S}=\{bw\in\mathcal{T}(G,S)\mid b\in B\land w\in (EG)^*\}=\bigsqcup_{b\in B}\mathcal{T}(G,S)_{b}\]
        so if the basis is finite the subspace is also cofinite.\\
        Furthermore if $\mathcal{S}$ is cofinite take $F\subseteq \mathcal{S}$ finite s.t.
            \[\mathcal{S}=\bigcup_{x\in F}\mathcal{T}(G,S)_x\] 
        thus for any $v\in \mathcal{S}$ we have some $x\in F$ s.t.~$x\leq_p v$. Therefore we can observe that $B\subseteq F$, since for each $b\in B$ we have some $x\in F\subseteq \mathcal{S}$ with $x\leq_p b$ and thus $b=x$. Thus since $F$ is finite, $B$ must also be finite.\\
        For the second part we first note that for any two subspaces $\mathcal{S}_1$ and $\mathcal{S}_2$ that have the same basis $B$ we must have:
        \[\mathcal{S}_1=\{bw\in\mathcal{T}(G,S)\mid b\in B\land w\in \mathcal{W}(G,T(b))\}=\mathcal{S}_2\]
        giving us uniqueness.\\
        Now for any set $B$ s.t.~no two distinct elements of $B$ are prefixes of each other we can define:
        \[\mathcal{S}_B:=\{bw\in\mathcal{T}(G,S)\mid b\in B\land w\in \mathcal{W}(G,T(b))\}\]
        this is clearly a subspace of $\mathcal{T}(G,S)$ and has $B$ as basis.
    \end{proof}
    Now we can also see that for any almost isomorphic trees we can find an almost isomorphism that is an isomorphism of two inescapable, cofinite subspaces, that also maps the basis of one subspace to the basis of the other. This follows from \lemref{lem2} since any isomorphism between subspaces will map basis element to a basis element since they are the only elements with empty incoming neighbourhood (in $\mathcal{S}$).\\
    Furthermore we can show that for any half graph, its unfolding tree is isomorphic to a half tree of the unfolding tree of the whole graph.
    \begin{lemma}\label{lem4}
        For any rooted graph $(G,S)$ any $x\in VG$ and any word $w\in\mathcal{W}(G,S)$ s.t.~$T(w)=x$ we have:
        \[\mathcal{T}(G_x,x)\cong_R \mathcal{T}(G,S)_{w}\]
    \end{lemma}
    \begin{proof}
        Firstly we have to note that any path in $G$ that has origin $x$ must be contained in $G_x$.\\
        We will define a graph homomorphism $\phi:\mathcal{T}(G_x,x)\to\mathcal{T}(G,S)_{w}$ by defining it on vertices as follows:
        \[\forall u\in \mathcal{W}(G_x,x),\ \phi(u):=wu\]
        we note that since any vertex in $\mathcal{T}(G,S)_{w}$ is of the form $wv$, where $v$ is a path with $O(v)=T(w)=x$ and thus $v\in \mathcal{W}(G_x,x)$. Conversely any word of such form is in $\mathcal{T}(G,S)_{w}$. This shows that $\phi$ is well-defined and a bijection on the vertices. \\
        Additionally we see that we have $(u_1,u_2)\in E\mathcal{T}(G,S)$ iff $(\phi(u_1),\phi(u_2))\in E\mathcal{T}(G,S)$ since for any $u\in \mathcal{W}(G_x,x)$ and any $e\in EG$ with $T(u)=o(e)$ we have $\phi(ue)=wue=\phi(u)e$, the converse follows from the bijectivity of $\phi$. So we can define $\phi$ on edges by setting:
        \[  \forall u\in \mathcal{W}(G_x,x)\forall e\in E_o(T(u)),\ \phi((u,ue))=(\phi(u),\phi(ue))\]
        This clearly makes $\phi$ into a tree homomorphism and since it is a bijection on edges and vertices it is a tree isomorphism. The fact that $\phi(\epsilon_x)=w$ shows us that $\phi$ is rooted.
    \end{proof}
    If we combine the above lemma with the fact that a cofinite inescapable subspace is a finite disjoint union of half trees s.t.~each path eventually lands in one of these trees, we can extrapolate that:
    \begin{lemma}\label{lem5}
        Let $(G,S),(H,R)$ be two rooted graphs, s.t.~the unfolding trees $\mathcal{T}(G,S),\mathcal{T}(H,R)$ are almost isomorphic then there exists subsets $M\subseteq VG$ and $M'\subseteq VH$ s.t.~there exists bases of cofinite inescapable spaces $B$, $B'$ in $\mathcal{T}(G,S),\mathcal{T}(H,R)$ resp. s.t.~$M=T[B]$ and $M'=T[B']$ and we have:
            \[\forall m\in M,\exists m'\in M',\ \mathcal{T}(G_m,m)\cong_R \mathcal{T}(H_{m'},m')\]
        If both $G$ and $H$ are non-redundant we furthermore have a bijection $\pi:M\to M'$ s.t.
            \[\forall m\in M G_m\cong_R H_{\pi(m)}\]
    \end{lemma}
    \begin{proof}
        Since $\mathcal{T}(G,S),\mathcal{T}(H,R)$ are almost isomorphic by \lemref{lem2} we have two cofinite inescapable subspaces $\mathcal{S}\subseteq\mathcal{T}(G,S)$ and $\mathcal{S}'\subseteq\mathcal{T}(H,R)$ and an isomorphism $\Phi: \mathcal{S}\to \mathcal{S}'$. Now let $B$ and $B'$ be the bases of $\mathcal{S}$ and $\mathcal{S}'$ resp., we note that $\Phi[B]=B'$. Note that we have:
        \[\mathcal{S}=\bigsqcup_{b\in B} \mathcal{T}(G,S)_b\qquad \mathcal{S}'=\bigsqcup_{b'\in B'} \mathcal{T}(H,R)_{b'}\]
        with the union being disjoint. Furthermore since the isomorphism must respect the ordering $\leq_p$ we must have for each $b\in B$:
        \begin{multline*}
        \Phi(\mathcal{T}(G,S)_b)=\Phi(\{ w\in\mathcal{T}(G,S)\mid b\leq_p w \})=\\
         \{ w'\in\mathcal{T}(H,R)\mid \Phi(b)\leq_p w' \}=\Phi(\mathcal{T}(H,R)_{\Phi(b)})
        \end{multline*}
        So $\mathcal{T}(G,S)_b\cong_R\mathcal{T}(H,R)_{\Phi(b)}$ via $\Phi|_{\mathcal{T}(G,S)_b}$.\\
        So by setting $m=T(b)\in VG$ and $m'=T(\Phi(b))\in VH$ we have by \lemref{lem4}:
        \[ \mathcal{T}(G_m,m)\cong_R\mathcal{T}(G,S)_{b}\cong_R \mathcal{T}(H,R)_{\Phi(b)}\cong_R \mathcal{T}(H_{m'},m')\]
        So by setting $M:=T[B]$ and $M':=T[B']$ we get the first part of the lemma since for any $m\in M$ we have some $b_m\in B$ s.t.~$m=T(b_m)$ and then $m'=T(\Phi(b_m))\in M'$ is the element we search.\\
        For the second part of the lemma assume that $G$ and $H$ are both non-redundant. Firstly we note that then for each $x\in VG$ $G_x$ is also non-redundant, since if we have some non edge collapsing equivalence relation on $G_x$, we can trivially expand it onto all of $G$ and it remains non edge collapsing since for each $y\in VG_x$: $E_o(y)\subseteq G_x$. So any such relation has to be trivial as $G$ is non-redundant. Analogously for each $x\in VH$ $H_x$ is also non-redundant. So if we have $M:=T[B]$ and $M':=T[B']$ as before we first any two $b_1,b_2\in B$ we have $T(b_1)=T(b_2)$ if and only if $T(\Phi(b_1))=T(\Phi(b_2))$ since because both $G$ and $H$ are non-redundant (by \lemref{lem0} part 1 and \lemref{lem4}):
        \begin{align*}
            T(b_1)=T(b_2)&\iff \mathcal{T}(G,S)_{b_1}\cong_R\mathcal{T}(G,S)_{b_2}\\
                         &\iff \mathcal{T}(H,R)_{\Phi(b_1)}\cong_R\mathcal{T}(G,S)_{\Phi(b_2)} \iff T(\Phi(b_1))=T(\Phi(b_2))    
        \end{align*}
        So the function $\pi:M\to M'$, defined by $\pi(m)=T(\Phi(b_m))$ where $b_m\in B$ is some element with $T(b_m)=m$, is well-defined and injective. It is surjective as for each $m'\in M$ there is some $b'_{m'}\in B'$ with $T(b'_{m'})= m'$ and since $\Phi[B]=B'$ we have some $b\in B$ s.t.~$\Phi(b)=b'_{m'}$ and thus $\pi(T(b))=T(b'_{m'})=m'$.\\
        For any $m\in M$ we can see by the definition of $\pi$ and by our previous consideration:
        \[\mathcal{T}(G_m,m)\cong_R \mathcal{T}(H_{\pi(m)},\pi(m))\]
        and since both $G_m$ and $H_{\pi(m)}$ is non-redundant by \coref{cor1} we must have:
        \[G_m\cong_R H_{\pi(m)}\]
    \end{proof}
    Now we can identify a family of graphs such that if two graphs of this family have almost isomorphic unfolding trees based on some root, then the graphs are isomorphic:
    \begin{defn}\label{def9}
        Let $G$ be a graph that has a root we will say that it is \textbf{robustly rooted} if for each $v\in VG$ with $G_v=G$ we have some $w\in N_o(v)$ s.t.~$G_w=G$.\\
        We will call a (unrooted) graph $G$ \textbf{robust} if either each inclusion maximal half-graph $G_v$ (i.e.~$G_v$ s.t.~$\forall w\in VG,\ G_v\subseteq G_w\implies G_v=G_w$) is robustly rooted or consists of just one vertex and no edges edges.
    \end{defn}
    Note that a graph is robustly rooted iff it has a root $v$ with $E_t(v)\neq \emptyset$.
    \begin{theorem}\label{thm2}
        For any finite non redundant robustly rooted graphs $G,H$ and any roots $S_G,S_H$ of $G,H$ resp. then  if $\mathcal{T}(G,S_G)$ is almost isomorphic to $\mathcal{T}(H,S_H)$ then $G\cong H$
    \end{theorem}
    \begin{proof}
        In order to show the theorem using \lemref{lem4} we will first show:
        \begin{clm}\label{clm1}
            For any finite robustly rooted graph $G$ and any root $S\in G$, we have for any basis $B\subseteq \mathcal{T}(G,S)$ of a cofinite inescapable subspace, some root $R$ of $G$ s.t.~$R\in T[B]$.
        \end{clm}
        \begin{proof}
             We prove this by induction over $s_{B}:=\sum_{b\in B}|b|$. We also define $m_B:=\max\{|b|\mid b\in B\}$.\\
            For the induction beginning we note that if $s_B=0$ then $m_B=0$ we must have $B=\{\epsilon_{S_G}\}$ and thus $M=\{S\}$, so $M$ contains $S_G$ which is a root.\\
            For the induction step, we will assume that $s_B>0$ and thus $m_B>0$, thus since $B$ is a basis we must have $\epsilon_{S}\notin B$. Let $b_m\in B$ be a path s.t.~$|b_m|=m_b$ since $\epsilon_{S_G}\notin B$, we must have $|b_m|>0$ and thus we have a path $w\in \mathcal{W}(G,S)$ and an edge $e\in EG$ with $o(e)=T(w)$ s.t.~$b_m=we$. Now for any other $e'\in EG$ s.t.~$o(e')=T(w)$ since $B$ is a basis $we'$ either has a proper prefix in $B$, is a proper prefix of some word in $B$, or is in $B$. If $we'$ has a proper prefix $b'<_p we'$ in $B$, then we must also have $b'\leq_p w$ and thus $b'<_p we$, but this contradicts the fact that $B$ is a basis. If $we'$ is a proper prefix of some $b'\in B$ then we must have $|b_m|=|we|=|we'|<|b'|$ but since $|b_m|$ is maximal in $\{|b|\mid b\in B\}$ this is a contradiction. Thus we must have $\forall e'\in E_o(T(w)),\ we'\in B$. So now if we define:
            \[B':=B\setminus\{we'\mid e'\in E_o(T(w))\}\cup\{w\}\]
            note that $s_{B'}<s_B$ and $m_{B'}\leq m_B$. To show that this is a basis take two elements $b'_1,b'_2\in B'$ s.t.~$b'_1\leq_p b'_2$. If we have $b'_1,b'_2\in B$ then since $B$ is a basis we must have $b'_1=b'_2$. If $b'_1=w$, since $|b'_1|\leq|b'_2|$ and $|b'_2|\leq m_{B'}\leq m_B=|w|+1=|b'_1|+1$ we must have $b'_2=w=b'_1$ or $b'_2=we$, for some $e\in E_o(T(w))$. However since $we\notin B$ for all $e\in E_o(T(w))$, we must have $b'_2=b'_1$. If $b'_2=w$ then we also have $b'_1\leq_p we$, so if $b'_1\in B$ we must have $b'_1=we\notin B'$, which is a contradiction and thus we must have $b'_1\notin B$ giving us $b'_1=w=b'_2$. Thus $B'$ is a basis. To show that $B'$ is a basis of a cofinite space  note that $B'$ is of course finite since $B$ is. Furthermore to show inescapability  we note that any vertex in $\mathcal{T}(G,S)$ is $\leq_p$- comparable with some element in $B'$ or with some element of the form $we'$, however if it is comparable with some $we'$ it is also comparable with $w\in B'$. So $B'$ is a basis of an inescapable cofinite space with $s_{B'}< s_B$ so by induction we must have for some root $R$ of $G$ s.t.:
            \[R\in T[B']=T[B]\setminus N_o(T(w))\cup \{T(w)\} \]
            If $R\in T[B]$ we are done, otherwise $R=T(w)$ but then since $G$ is robustly rooted we must have some other root $R'\in N_o(T(w))\subset T[B]$. Thus showing that $T[B]$ contains a root.\\ 
        \end{proof}
        Now take $G,H$ to be finite robustly rooted non redundant with roots $S_G,S_H$ s.t.~the trees $\mathcal{T}(G,S_G)$ and $\mathcal{T}(H,S_H)$ are almost isomorphic. Thus using \lemref{lem5}, we have some bases $B_G,B_H$ cofinite inescapable subspaces and a bijection $\pi:M_G\to M_H$, where $M_G=T[B_G]$ and $M_H=T[B_H]$, s.t.~for each $m\in M_G$ $G_m\cong H_{\pi(m)}$. However by the above claim we have some roots $R_G\in M_G$ and $R_H\in M_H$ we have $G_{R_G}=G\cong G_{\pi(R_G)}$ and $G_{\pi^{-1}(R_H)}\cong H=H_{R_H}$. And since isomorphic graphs have the same amount of vertices and any subgraph has less or equally many vertices than the graph it is contained in we have:
        \[|VG|=|VH_{\pi(R_H)}|\leq |VH|= |VG_{\pi^{-1}(R_H)}|\leq |VG|\]
        Thus we have $|VH_{\pi(R_H)}|=|VH|$, but since $H_{\pi(R_H)}$ is a vertex induced subgraph of $H$ we must thus have $H_{\pi(R_H)}=H$ and so $G\cong H$. 
    \end{proof}
    Note that the converse of the above theorem is not necessarily true since sometimes for two distinct roots $S\neq R$ of a graph $G$, the trees $\mathcal{T}(G,S)$ and $\mathcal{T}(G,R)$ are not almost isomorphic , see section 5. We will show later how to identify when two roots produce almost isomorphic unfolding trees.\\
    In order to get a better understanding of when two not necessarily robustly rooted graphs produce almost isomorphic unfolding trees we introduce the \textbf{spider product} of several graphs:
    \begin{defn}\label{def11}
        For any graphs $G_1,G_2,\dots,G_n$, any $k_1,k_2,\dots k_n\in \mathds{Z}_{>0}$ and any functions $\rho_1,\rho_2,\dots,\rho_n$:
        \[\forall 1\leq i\leq n,\ \rho_i:\{1,\dots,k_i\}\to VG_i\]
        We can define the spider product $(G,S):=\llangle(G_1,k_1,\rho_1),\dots,(G_n,k_n,\rho_n) \rrangle$ as follows:
        \begin{align*}
            VG:&=\{S\}\sqcup \bigsqcup_{1\leq i\leq n} VG_i\\
            EG:&=\{(S,\rho_i(j),j)\mid 1\leq i\leq n,1\leq j \leq k_i\}\sqcup \bigsqcup_{1\leq i\leq n} EG_i\\
            \forall 1\leq i\leq n,& \forall e\in EG_i,\ (o_G(e):= o_{G_i}(e))\land (t_G(e):= t_{G_i}(e))\\
            \forall 1\leq i\leq n,&\forall 1\leq j \leq k_i,\ (o((S,\rho_i(j),j)):= S)\land (t((S,\rho_i(j),j)):= \rho_i(j))  
        \end{align*}
    \end{defn}
    \begin{figure}
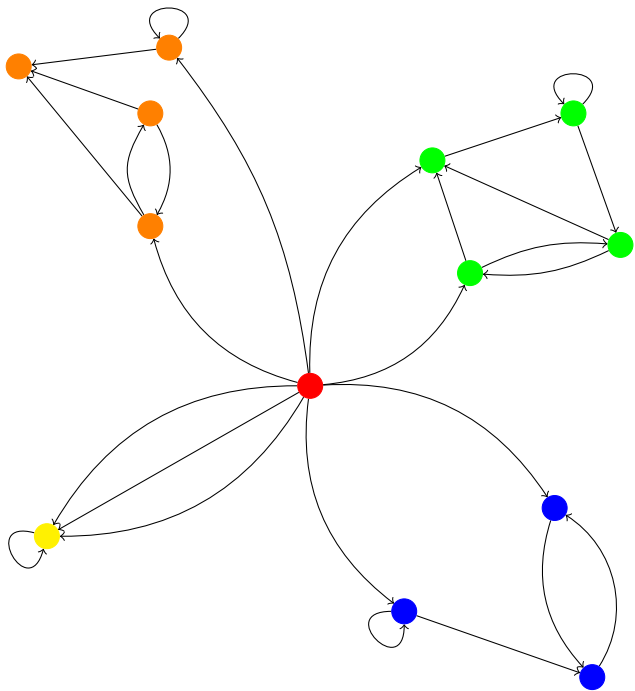

    \ctikzfig{spider_product}
    \caption{An example of a spider product}
    \end{figure}
    In order for a spider product $\llangle(G_1,k_1,\rho_1),\dots,(G_n,k_n,\rho_n) \rrangle$  to be a rooted graph with root $S$ we must have for each $i\in\{1,\dots,n\}$:
    \[\bigcup_{j=1}^{k_i} (G_i)_{\rho(j)}=G_i\]
    We will assume this from now on unless stated otherwise.\\

    We adopt the convention that if $k_i=0$ for some $1\leq i\leq n$ we simply drop that component of the spider product i.e.
    \begin{multline*}
        \llangle(G_1,\rho_1,k_1),\dots, (G_i,\rho_i,0),\dots(G_n,\rho_n,k_n)\rrangle:=\\
        \llangle(G_1,\rho_1,k_1),\dots, (G_{i-1},\rho_{i-1},k_{i-1}),(G_{i+1},\rho_{i+1},k_{i+1}),\dots(G_n,\rho_n,k_n)\rrangle
    \end{multline*}
    
    Note that permuting the components of the spider product does not change its structure. Note as well that we can take the union of the components and not change the structure, i.e.:
    \begin{multline*}
        \llangle(G_1,\rho_1,k_1),\dots, (G_i,\rho_i,k_i), (G_{i+1},\rho_{i+1},k_{i+1})\dots(G_n,\rho_n,k_n)\rrangle:=\\
        \llangle(G_1,\rho_1,k_1),\dots, (G_{i}\sqcup G_{i+1},\rho_{i}\oplus\rho_{i+1} ,k_i+k_{i+1}),\dots(G_n,\rho_n,k_n)\rrangle
    \end{multline*}
    Here for any $\gamma:\{1,\dots,k\}\to X$ and $\delta:\{1,\dots,l\}\to Y$ the function $\gamma\oplus \delta:\{1,\dots,k+l\}\to X\cup Y$ is defined by setting:
    \[\gamma\oplus\delta(i):=\begin{cases}
                                \gamma(i),& \text{if } 1\leq i\leq k \\
                               \delta(i-k),& \text{if } k<i\leq k+l
                            \end{cases}\]
    So when we look at a spider product we can assume that the components are connected.\\
    Our first observation about spider products is that factoring a spider product over a non-edge collapsing equivalence relation produces another spider product.
    \begin{lemma}\label{lem9}
        For any spider product $(G,S)=\llangle(G_1,k_1,\rho_1),\dots,(G_n,k_n,\rho_n) \rrangle$ and any non edge collapsing equivalence relation $\sim$, s.t. $[S]_{\sim}=\{S\}$ there exists another spider product $(H,R)$ s.t.:
            \[(G/{\sim},[S]_{\sim})\cong_R(H,R)\]
    \end{lemma}
    \begin{proof}
        We will first note that for any vertex $v$ s.t.~$S\sim v$ we must have $v=S$.\\
        Thus if we define the subgraph $\bar{S}:=(\{S\}\cup N_o(S),E_o(S))$ we can restrict the equivalence relation $\sim$ to $G\setminus \bar{S}$ and define $H_0:=(G\setminus \bar{S})/{\sim}$.
        Take $k_0:=k_1+k_2+\dots+k_n$ and $\rho_0:\{1,\dots,k_0\}\to VH_0$ to be a function defined by:
        \[\forall 1\leq i\leq n, \forall 1\leq l\leq k_i,\, \rho_0(\sum_{1\leq j<i}k_j+l)=\rho_i(l)\]
        To show that $(G/{\sim},[S]_{\sim})\cong_R\llangle (H_0,k_0,\rho_0) \rrangle$ we will define the function:
        \[\forall x\in EG/{\sim}\sqcup VG/{\sim},\, \Phi(x):=   \begin{cases}
                                                                    R, &  x=[S]_{\sim}\\
                                                                    (R,\rho_0(\sum_{1\leq j<i}k_j+l),\sum_{1\leq j<i}k_j+l),& x=[(S,\rho_i(l),l)]\\
                                                                    x, & \text{else}
                                                                \end{cases}\]
        Since as we have noted that $\sim$ is trivial on $\bar{S}$ this is well-defined and an isomorphism.
    \end{proof}
    We will by abuse of notation call any spider product non-redundant if it does not admit a non-trivial equivalence relation as in the above lemma. This is not always the same as being generally non-redundant, but it means that any half graph not equal to the whole graph is non-redundant.\\
    We can show that non-isomorphic, non-redundant spider products spider products, have non-isomorphic unfolding trees. This is because if we have two non-redundant spider products $(G,S)$ and $(H,R)$ that have the same unfolding trees, then we have some non-edge-collapsing equivalence relations $\sim_G$ and $\sim_H$ s.t. $(G/{\sim_G},[S]_{\sim_G})\cong_R(H/{\sim_H},[R]_{\sim_H})$. However since they are non-redundant spider-products $\sim_G$ and $\sim_H$ have to be trivial when restricted to the vertices that are not the root. So they are either trivial or relates the root to one non-root vertex and is trivial on all other vertices. If they are both trivial then of course $(G,S)\cong_R(H,R)$. It is impossible for just one of the relations be non-trivial, since if we have $S\sim_G v$ for some $v\neq S$, we have $N_t([S]_{\sim_G})=N_t([v]_{\sim_G})\neq \emptyset$, since $N_t(v)\neq \emptyset$, as there exists a path from $S$ to $v$. However since $N_t(R)=\emptyset$ there cannot be an isomorphism between $H$ and $G/{\sim_G}$ mapping $R$ to $[S]_{\sim_G}$. This contradicts $(G/{\sim_G},[S]_{\sim_G})\cong_R(H/{\sim_H},[R]_{\sim_H})$.\\
    If both of the equivalence relations are non-trivial, let $\phi_{\sim}:G/{\sim_G}\to H/{\sim_H}$ be the rooted isomorphism we can define the $\phi$ to be the isomorphism from $G$ to $H$. We can define $\phi(S)=R$ and for any vertex $v\neq S$: $\phi(v)$ to be the unique vertex s.t. $[\phi(v)]_{\sim_H}=\phi([v]_{\sim_G})$. Analogously we can define the edge component of $\phi$. This gives us a rooted isomorphism.\\   
    Now we can show that each unfolding tree of a graph $G$  is almost isomorphic to the unfolding tree of a spider product of rooted graphs which are either robust or consist of just one vertex  with no edges.
    \begin{lemma}\label{lem7}
        For any rooted graph $(H,R)$ there exists $k_1,\dots,k_n\in\mathds{Z}_{>0}$,robust graphs $G_1,\dots,G_n$ and functions $\rho_1,\dots,\rho_n$ as in the definition of the spider product.  s.t.~for: 
        \[(G,S):=\llangle (G_1,\rho_1,k_1),\dots, (G_n,\rho_n,k_n)\rrangle\]
        the unfolding tree $\mathcal{T}(H,R)$ is almost isomorphic to $\mathcal{T}(G,S)$.   
    \end{lemma}
    \begin{proof}
        In order to show the lemma we need to first establish for which vertices $x$ the graph $H_x$ is robustly rooted.
        \begin{clm}\label{clm2}
            For any vertex $x\in VH$ s.t.~there exists a non-empty path $p\in \mathcal{W}(H)$ with $O(p)=x=T(p)$ the half-graph $H_x$ is robustly rooted.
        \end{clm}
        \begin{proof}
            This follows immediately from the fact that in $H_x$, $x$ is a root with non-empty incoming neighbourhood (i.e. $N_t(v)\neq \emptyset$).
        \end{proof}
        Now define:
        \[M_1=\{x\in VH\mid \exists p\in \mathcal{W}(H),\ (p\neq \epsilon_x)\land (O(p)=x=T(p))\}\]
        and:
        \[M_2=\{x\in VH\mid E_o(x)=\emptyset\}\]
        and set $M:=M_1\cup M_2$. Now we can define:
        \[\mathcal{S}_M=\{w\in \mathcal{W}(H,R)\mid \exists u\leq_p w,\ T(u)\in M \}\]One can see easily that this is a subspace since for any $w\in \mathcal{S}_M$ and any $w\leq_p w'$ we have some $u\in \mathcal{W}(H,R)$ with $T(u)\in M$ s.t.~$u\leq_p w\leq_p w'$ and thus $u\leq_p w'$ giving us $u\in \mathcal{S}_M$.\\
        Now to show that $\mathcal{S}_M$ is inescapable take some $u\in \mathcal{W}(H,R)$ and look at the set: $V\mathcal{W}(H,R)_u=\{w\in\mathcal{W}(H,R)\mid u\leq_p w\}$.\\
        If there is some word $w\in V\mathcal{W}(H,R)_u$ with $T(w)\in M_2$, then $u\leq_p w\in \mathcal{S}_M$.\\
        If that is not the case we will construct a series of paths $w_1,w_2,\dots\in V\mathcal{W}(H,R)_u$ s.t.~$w_1\leq_p w_2\leq_p w_3\leq_p\dots$. We can define this series recursively by taking $w_1:=u$ and when we have already defined $w_i$ for some $i\in\mathds{N}$ then since $w_i\in V\mathcal{W}(H,R)_u$ we have by our previous assumption we have $T(w_i)\notin M_2$ and thus take some $e_i\in E_o(w_i)$ then if we define $w_{i+1}=w_ie_i$ we have $w_i\leq_p w_{i+1}$. This constructs our series, now take the vertices $T(w_1),T(w_2),\dots \in VG$ and since $G$ is a finite graph we must have some $j<k\in \mathds{N}$ with $T(w_j)=T(w_k)$, let $u_{jk}\in\mathcal{W}(G)$ be s.t.~$w_k=w_ju_{jk}$. Note that $u_{jk}$ is non empty and $O(u_{jk})=T(w_j)=T(w_k)=T(u_{jk})$, thus we must have $T(w_j)\in M_1$ and thus $u\leq_p w_j\in \mathcal{S}_M$.\\
        Now in order to show that $\mathcal{S}_M$ is cofinite we will note that the set:
        \[B_M:=\{w\in \mathcal{T}(H,R)\mid (T(w)\in M)\land (\forall u<_p w,\ T(u)\notin M)\}\]
        is the basis of the subspace. To show that it is finite we note that for any $u\in \mathcal{T}(H,R)$ with $m:=|u|>|VH|+1$ we can write $u=e_1e_2\dots e_m$ and note that the vertices $t(e_1),\dots,t(e_{m-1})\in VH$ cannot all be distinct vertices. Take thus $j<k< m$ s.t.~$t(e_j)=t(e_k)$ so by setting $w=e_1\dots e_k<_p u$ we see that $T(w)\in M_1$ since $e_{j+1}\dots e_k$ is a non-empty cycle starting  and ending in $T(w)$. So since $w<_p u$ we have $u\notin B_M$. Thus we must have $B_M\subseteq\{w\in\mathcal{T}(H,R)\mid |w|\leq |VH|\}$ which is a finite set. This gives us the cofiniteness of $\mathcal{S}_M$.\\
        Now let $w_1,w_2,\dots,w_n\in \mathcal{W}(H,R)$ be paths s.t.~$B_M=\{w_1,w_2,\dots,w_n\}$. Now we can define the functions $\rho_i:\{1\}\to H_{T(w_i)}$ for each $i\in \{1,\dots,n\}$ by setting $\rho_i(1)=T(w_i)$. Now we can set:
        \[(G,S):=\llangle (H_{T(w_1)},T(w_1),1),\dots (H_{T(w_n)},T(w_n),1)\rrangle\]
        and now we can define the subspace of $\mathcal{T}(G,S)$ to be:
        \[\mathcal{S}:=\{u\in\mathcal{T}(G,S)\mid \exists 1\leq i\leq n,\ T(u)\in VH_{T(w_i)}\}\]
        We note that:
        \[\mathcal{S}=\bigsqcup_{1\leq i\leq n} T(G,S)_{(S,T(w_i),1)}\]
        and thus it is a cofinite subspace and since each word in $\mathcal{W}(G,S)$ that has length more than $1$ is in $\mathcal{S}$.\\
        Now by \lemref{lem4} we have 
        \[\mathcal{S}=\bigsqcup_{1\leq i\leq n} \mathcal{T}(G,S)_{(S,T(w_i),1)}\cong \bigsqcup_{1\leq i\leq n} \mathcal{T}(G_{T(w_i)},T(w_i))\]
        but since $G_{T(w_i)}=H_{T(w_i)}$ for each $1\leq i\leq n$ we have:
        \begin{multline*}
            \bigsqcup_{1\leq i\leq n} \mathcal{T}(G_{T(w_i)},T(w_i))=\bigsqcup_{1\leq i\leq n}\mathcal{T}(H_{T(w_i)},T(w_i))=\\
            \bigsqcup_{w\in B_M} \mathcal{T}(H_{T(w)},T(w))\cong \bigsqcup_{w\in B_M} \mathcal{T}(H,R)_w=\mathcal{S}_M
        \end{multline*}
        Thus $\mathcal{S}\cong \mathcal{S}_M$ and therefore the trees $\mathcal{T}(H,R)$ and $\mathcal{T}(G,S)$ are almost isomorphic.
    \end{proof}
    Thanks to the above lemma if we want to look at almost automorphism classes of label regular trees we can just look at unfolding trees of spider products, where the components are robust . When looking at graphs without sinks, one can assume that none of the components of the product is one vertex with no edges. We can furthermore assume that these products are non-redundant by using \Lemref{lem9}.

    Especially we can assume that the components of the product are non-redundant. We will define for any graph $G$ the vertex subset:
    \[ VG^m:=\{v\in VG\mid G_v\text{ is inclusion maximal}\} \]
     To show that if we have two non-redundant spider products that produce almost isomorphic unfolding trees, then the components of the products are isomorphic, we will need a basic result about non-redundant graphs:
    \begin{lemma}\label{lem10}
        Let $G,H$ be two non-redundant graphs s.t.
        \begin{align*}
            \forall& v\in VG^m\exists v'\in VH,\ G_v\cong_R H_{v'}\\
            \forall& w'\in VH^m\exists w\in VG,\ H_{w'}\cong_R G_{w}\\
        \end{align*}
        Then we must have:
        \[G\cong H\]
    \end{lemma}
    \begin{proof}

        To construct an isomorphism between $G$ and $H$ we will look at the graph $F:=G\sqcup H$ (assuming w.l.o.g that $G$ and $H$ are disjoint) and take $\sim$ to be a non edge collapsing equivalence relation with $\sim_V=\sim_o$. Since $G$ and $H$ are both non-redundant we must have:
        \begin{align*}
            \forall x_1,x_2\in VG\sqcup EG,\ (x_1\sim x_2)\iff (x_1=x_2)\\
            \forall y_1,y_2\in VG\sqcup EG,\ (y_1\sim y_2)\iff (y_1=y_2)
        \end{align*}
        so for each vertex $v$ in $G$ we have at most one $v'\in VH$ with $v\sim_o v'$ (otherwise we would violate the above statements via transitivity). To see that such a $v'$ actually exists we look at some $u\in VG^m$ s.t. $G_v\subseteq G_u$, since we have some $u'\in VH$ s.t.~a rooted isomorphism $\phi:G_u\to H_{u'}$ exists. Thus we have some $v'\in H_{u'}$ s.t.
        \[F_{v'}=H_{v'}= \phi(G_v)\cong_R G_v=F_v\]
        so we must have $v\sim_o v'$ (i.e. $v\sim v'$).\\
        So there is a injection $\pi_V:VG\to VH$ s.t.~$\forall v\in VG v\sim \pi(v)$ and since we can reverse this construction by swapping $G$ and $H$ it is in fact a bijection. Furthermore since $\sim$ is non edge collapsing for each $e\in EG$ there is a unique $e'\in E_o(\pi(o(e)))$ s.t. $e\sim e'$ and as $\pi(o(e))$ is the only element of $VH$ that is related by $\sim$ to $o(e)$, this is the unique $e'\in EH$ s.t. $e\sim e'$.
        So we get an injection $\pi_E:EG\to EH$ s.t.~$\forall e\in EG e\sim \pi(e)$, again since the above construction is reversible $\pi_E$ is in fact a bijection.\\
        So if we set $\pi=(\pi_V,\pi_E)$ it is a graph isomorphism between $G$ and $H$, using the fact that $\sim$ is a graph equivalence relation.
    \end{proof}
    Now we can show that if two spider products produce almost isomorphic unfolding trees, the disjoint unions of their components are isomorphic.
    \begin{lemma}\label{lem11}
        For any two non-redundant spider products with robust components:
        \begin{align*}
            (G,S):&=\llangle (G_1,\rho_1,k_1),\dots,(G_n,\rho_n,k_n) \rrangle\\
            (H,R):&=\llangle (H_1,\sigma_1,l_1),\dots,(H_m,\sigma_m,l_m) \rrangle  
        \end{align*}
        where $G_1,\dots G_n,H_1,\dots,H_m$ are pairwise disjoint, with:
        \[ \mathcal{T}(G,S)\cong_A \mathcal{T}(H,R) \]
        we must have:
        \[\bigsqcup_{i=1}^{n}G_i\cong\bigsqcup_{j=1}^{m}H_j  \]
    \end{lemma}
    \begin{proof}
        Let $\mathcal{S},\mathcal{R}$ be cofinite inescapable subspaces of $\mathcal{T}(G,S),\mathcal{T}(H,R)$ resp. such that we have a isomorphism:
        \[\Phi:\mathcal{S}\to \mathcal{R}\]
        Let $B,C$ be the bases of $\mathcal{S},\mathcal{R}$ respectively. Now take for any $1\leq i\leq n$ some $v\in N_o(S)$ and consider $\mathcal{T}(G,S)_{(S,v,k)}\cap B$ for any $k\leq k_i$ s.t.~such an edge exists.  Now take $\phi_{v}:\mathcal{T}(G,S)_{(S,v,k)}\to \mathcal{T}(G_v,v)=\mathcal{T}((G_i)_v,v)$ to be the isomorphism from the proof of \lemref{lem4}. We can see that $\phi_{v}(\mathcal{T}(G,S)_{(S,v,k)}\cap B)$ is a basis of an inescapable cofinite subspace and since $G_i$ is robust, by the first claim in the proof of \thmref{thm2} we have some $y\in\phi_{v}(\mathcal{T}(G,S)_{(S,v,k)}\cap B)$ s.t.~T(y) is a root of $(G_i)_v$. Now let $x\in \mathcal{T}(G,S)_{(S,v,k)}\cap B$ be s.t.~$\phi_v(x)=y$, then $T(x)=T(y)$ and $G_{T(x)}=(G_i)_{T(y)}=(G_i)_v$. Now since $\mathcal{T}(G,S)_x\subseteq \mathcal{S}$, we can look at $\Phi(\mathcal{T}(G,S)_x)$. Since $\Phi$ is an isomorphism this will also be a half tree, rooted in some point in the basis so take $z\in C$ to be s.t.~$\Phi(\mathcal{T}(G,S)_x)=\mathcal{T}(H,R)_z$ so again we can apply \lemref{lem4} and the fact that $H$ and $G$ are non-redundant to get $G_{T(x)}\cong H_{T(z)}$. Taking $j$ to be s.t.~$T(z)\in H_j$ we get:
        \[(G_i)_v=G_{T(x)}\cong H_{T(z)}=(H_j)_{T(z)}\]
        Now for any $w\in VG_i^m$ we have some $v\in N_o(S)$ s.t. $(G_i)_w=(G_i)_v$ and thus we have some $H_j$ and some $u\in H_j$ s.t.
            \[(G_i)_v=(G_i)_w\cong(H_j)_u\]
        and by changing the root of $(H_j)_u$ to the image of $v$ we may assume that we have some $u'\in H_j$ s.t.
            \[(G_i)_v\cong_R(H_j)_{u'} \] 
        So each inclusion maximal half graph of each $G_i$ is rooted isomorphic to some half graph of some $H_j$.\\
        By reversing the preceding argument(we can do this by exchanging $G$ and $H$ and substituting $\Phi^{-1}$ for $\Phi$) we also get that each inclusion maximal half graph of each $H_j$ is isomorphic to some half graph of some $G_i$.\\
        So when we define $\tilde{G}:=\bigsqcup_{i=1}^{n}G_i$ and $\tilde{H}:=\bigsqcup_{j=1}^{m}H_j$ we can note that:
        \begin{align*}
            V\tilde{G}^m&=\bigsqcup_{i=1}^n VG_i^m\\
            V\tilde{H}^m&=\bigsqcup_{j=1}^m VH_j^m
        \end{align*}
        So using \lemref{lem10} we get:
        \[\tilde{G}\cong \tilde{H}\]

    \end{proof}
    So splitting every component into their connected parts and rearranging them we may assume that any two spider products $(G,S),(H,R)$ with robust components that produce almost isomorphic unfolding trees are are of the form:
      \begin{align*}
            (G,S):&=\llangle (G_1,\rho_1,k_1),\dots,(G_n,\rho_n,k_n) \rrangle\\
            (H,R):&=\llangle (G_1,\sigma_1,l_1),\dots,(G_n,\sigma_m,l_m) \rrangle  
        \end{align*}
    with each $G_i$ being connected and robust.\\
    In order to answer when two such spider products are almost isomorphic we have to introduce the \textbf{graph monoid} (following the definition from \cite{ara2007}). 
    \begin{defn}
         For any graph $G$ the graph monoid $\mathcal{M}(G)$ is the commutative monoid generated by the vertices $VG$, with the relations:
         \[\forall v\in VG\text{ s.t.~} E_o(v)\neq \emptyset,\ v=\sum_{e\in E_o(v)}t(e)\]
    \end{defn} 
    Now we can connect the structure of the monoid to the structure of subspaces of the unfolding trees and their bases.
    \begin{lemma}\label{lem12}   
        For any graph $G$ and any root $S\in VG$ we have for each basis $B$ of an inescapable cofinite subspace of $\mathcal{T}(G,S)$ we have:
        \[S=\sum_{p\in B} T(p)\]         
    \end{lemma}
    \begin{proof}
        We will prove the lemma by induction over $s_B:=\sum_{p\in B}|p|$. From the proof of \thmref{thm2} we see that there is some path $p\in\mathcal{T}(G,S)$, with $ E_o(T(p))\neq \emptyset$ s.t.~$\{pe\mid o(e)=T(p)\}\subseteq B$ then:
        \[B':=B\setminus \{pe\mid o(e)=T(p)\}\cup \{p\}\]
        is also a basis of a cofinite inescapable subspace with $s_{B'}<s_B$, and thus by induction we have in $\mathcal{M}(G)$:
        \[S=\sum_{q\in B'}T(q)=T(p)+\sum_{q\in B'\setminus\{p\}}T(q)\]
        and by the definition of the graph monoid we have:
        \[S=T(p)+\sum_{q\in B'\setminus\{p\}}T(q)=\sum_{e\in E_o(T(p))}T(pe)+\sum_{q\in B'\setminus\{p\}}T(q)=\sum_{q\in B}T(q)\]
    \end{proof}
    A converse of the above lemma can be formulated as follows:
    \begin{lemma}\label{lem13}
        For any graph $G$ and any two roots $S,R$ of $G$ s.t.~in $\mathcal{M}(G)$:
        \[S=R\]
        there is a basis $B_S$ of an cofinite inescapable subspace in $\mathcal{T}(G,S)$ and a basis $B_R$ of an cofinite inescapable subspace in $\mathcal{T}(G,R)$, with a bijection:
        \[\pi: B_S\to B_R\]
        s.t.~$\forall p\in B_S,\ T(p)=T(\pi(p))$    
    \end{lemma} 
    \begin{proof}
        To show this we will look at the free commutative monoid over $VG$: $\mathcal{F}(G)$ and define the set $R_0\subseteq\mathcal{F}(G)\times \mathcal{F}(G)$ by:
        \[R_0:=\{\Big(s+v,s+(\sum_{e\in E_o(v)}t(e))\Big)\in \mathcal{F}(G)\times\mathcal{F}(G) \mid s\in \mathcal{F}(G),v\in VG,E_o(v)\neq \emptyset\}\]
        In \cite[Lemma~4.3]{ara2007} it has been shown that if two formal sums $s,s'\in \mathcal{F}(G)$ to be equal in $\mathcal{M}(G)$ we have two finite series: $s=s_0,s_1,\dots,s_n\in \mathcal{F}(G)$ and $s'=s'_0,s'_1,\dots,s'_m\in \mathcal{F}(G)$ s.t.~$(s_0,s_1),\dots,(s_{n-1},s_{n}),(s'_0,s'_1),\dots,(s'_{m-1},s_m)\in R_0$ and $s_n=s'_m$ in $\mathcal{F}(G)$.\\
        So we will need to show that for any such sequence, if $s_0=S$ then $s_n=\sum_{p\in B} T(p)$, for some basis of an inescapable cofinite subspace. This can be done with induction over $n$.\\
        For the induction beginning at $n=0$ we will just take $B=\{\epsilon_S\}$.\\
        For the induction step take $n>0$ and by induction hypothesis we have some basis of an inescapable cofinite subspace $B'$, s.t.~$s_{n-1}=\sum_{p'\in B'} T(p')$. So by the definition of $R_0$ we have some $q\in B'$, with $E_o(T(q))\neq \emptyset$ s.t.:
        \[s_n=\sum_{p'\in B'\setminus\{q\}}T(p')+\sum_{e\in o^{-1}(\{T(q)\})}t(e)=\sum_{p\in B'\setminus\{q\}\cup\{qe\mid o(e)=T(q)\}}T(p)\]
        Now if we define $B:=B'\setminus\{q\}\cup\{qe\in\mathcal{T}(G,S)\mid o(e)=T(q)\}$, we can see that this is a basis of a cofinite inescapable subspace, by the fact that $B'$ is. This shows the claim.\\
        So if we have two roots $S,R$ in $G$ s.t.~$S=R$ in the graph monoid we have a basis $B_S$ of an cofinite inescapable subspace in $\mathcal{T}(G,S)$ and a basis $B_R$ of an cofinite inescapable subspace in $\mathcal{T}(G,R)$ s.t.~in $\mathcal{F}(G)$:
        \[\sum_{p\in B_S}T(p)=\sum_{p\in B_R}T(p)\]
        By definition this means that we have a bijection:
        \[\pi: B_S\to B_R\]
        s.t.~$\forall p\in B_S,\ T(p)=T(\pi(p))$ as required.
    \end{proof}
    Now we can connect the graph monoid and almost isomorphism:
    \begin{theorem}\label{thm3}
        For any non-redundant graph $G$ and any two roots $S,R$ the unfolding trees $\mathcal{T}(G,S)$, $\mathcal{T}(G,S)$ are almost isomorphic if and only if:
        \[S=R\]
        in the graph monoid.
    \end{theorem}
    \begin{proof}
        For the ``only if'' direction assume that  $\mathcal{T}(G,S)$, $\mathcal{T}(G,S)$ are almost isomorphic and take some cofinite subspaces $\mathcal{S}\subseteq \mathcal{T}(G,S)$, $\mathcal{S}'\subseteq \mathcal{T}(G,R)$ with an isomorphism:
        \[\Phi:\mathcal{S}\to\mathcal{S}'\]
        Let $B,B'$ be the bases of $\mathcal{S}$ and $\mathcal{S}'$ resp. Now since $\Phi[B]=B'$ this gives us a bijection $\pi:=\Phi|_B$ between $B$ and $B'$. Furthermore by \lemref{lem4} we have  $\forall b\in B,\ \mathcal{T}(G_{T(b)},T(b))\cong_R \mathcal{T}(G_{T(\pi(b))},T(\pi(b)))$ and thus since $G$ is non-redundant we must have $T(b)=T(\pi(b))$. So we have:
        \[\sum_{p\in B} T(p)=\sum_{p'\in B'} T(p')\]
        in the graph monoid. \Lemref{lem12} gives us:
         \[S=\sum_{p\in B} T(p)=\sum_{p'\in B'} T(p')=R\]
         For the converse implication we take from \lemref{lem13} some cofinite inescapable subspaces $\mathcal{S}\subseteq \mathcal{T}(G,S)$, $\mathcal{S}'\subseteq \mathcal{T}(G,R)$ s.t.~for it's resp. bases $B,B'$ we have a bijection:
         \[\pi: B\to B'\]
        s.t.~$\forall p\in B_S,\ T(p)=T(\pi(p))$. So by \lemref{lem4} we have $\mathcal{T}(G,S)_p\cong\mathcal{T}(G,R)_{\pi(p)}$. Since these half-trees are the connected components of $\mathcal{S}$ and $\mathcal{S}'$ this gives us an isomorphism between them. So the trees $\mathcal{T}(G,S)$ and $\mathcal{T}(G,R)$ are almost isomorphic.\\
    \end{proof} 
    If we combine this theorem with \thmref{thm2} we can classify unfolding trees of robustly rooted graphs. This classification simply associates each such tree with a unique tuple consisting of a non-redundant robustly rooted graph and a member of the graph monoid that is equal to some root of this graph.   
    To classify general unfolding trees we will have to generalise the above results.
    \begin{lemma}\label{lem14}
        For any graph $G$, any positive integer $k\in \mathds{Z}_{>0}$ and any function $\rho:\{1,\dots,k\}\to VG$ s.t.
        \[\bigcup_{1\leq i\leq k} G_{\rho(i)}=G\]
        we have for each basis $B$ of some cofinite inescapable subspace of $\mathcal{T}(\llangle(G,k,\rho)\rrangle,S)$:
        \[\sum_{i=1}^k\rho(i)=\sum_{b\in B} T(b)\]
        in $\mathcal{M}(G)$
    \end{lemma}
    \begin{proof}
        We first note that if any two sums $s,t\in \mathcal{F}(G)$ are equal in $\mathcal{M}(\llangle(G,k,\rho)\rrangle)$ then from \cite[Lemma~4.3]{ara2007} we have $s=s_0,s_1.\dots s_n, t=t_0,t_1,\dots,t_m\in \mathcal{F}(\llangle(G,k,\rho)\rrangle)$ s.t.~$s_n=t_m$ and 
        \[(s_0,s_1),\dots,(s_{n-1},s_n),(t_0,t_1),\dots,(t_{m-1},t_m)\in R_0(\llangle(G,k,\rho)\rrangle)\]
        Where $R_0$ is defined as in the proof of \Lemref{lem13}. However since $S$ has no incoming edges and $s,t$ have no $S$ in the sums neither can $s_1,\dots,s_n$ and $t_1,\dots,t_m$. Thus we have 
        \[(s_0,s_1),\dots,(s_{n-1},s_n),(t_0,t_1),\dots,(t_{m-1},t_m)\in R_0(G)\]
        giving us $s=t$ in $\mathcal{M}(G)$. So it is enough to show the equality from the Lemma in $\mathcal{M}(\llangle(G,k,\rho)\rrangle)$.\\
        To do that we note that in $\mathcal{M}(\llangle(G,k,\rho)\rrangle)$:
        \[S=\sum_{e\in E_o(S)}t(e)=\sum_{i=1}^k t( (S,\rho(i),i) )=\sum_{i=1}^k \rho(i)\]
        And by \lemref{lem12} we have in  $\mathcal{M}(\llangle(G,k,\rho)\rrangle)$:
        \[ S=\sum_{p\in B} T(p) \]
        giving us:
        \[\sum_{i=1}^k\rho(i)=\sum_{b\in B} T(b)\]
        in $\mathcal{M}(\llangle(G,k,\rho)\rrangle)$ and thus also in $\mathcal{M}(G)$.
    \end{proof}
    Similarly we can generalise \lemref{lem13} as follows:
    \begin{lemma}\label{lem15}
        For any non-redundant graph $G$, any two $k,l\in \mathds{Z}_{>0}$ and any two functions $\rho:\{1,\dots,k\}\to VG$, $\sigma:\{1,\dots,l\}\to VG$ s.t.
        \[\bigcup_{1\leq j\leq l} G_{\sigma(j)}=\bigcup_{1\leq i\leq k} G_{\rho(i)}=G\]
        if we have:
        \[\sum_{i=1}^k\rho(i)=\sum_{j=0}^l\sigma(j)\]
        in $\mathcal{M}(G)$, then for:
        \begin{align*}
            (G_{\rho},S_{\rho}):=\llangle(G,\rho,k)\rrangle\\
            (G_{\sigma},S_{\sigma}):=\llangle(G,\sigma,l)\rrangle 
        \end{align*} 
        the trees $\mathcal{T}(G_{\rho},S_{\rho}),\mathcal{T}(G_{\sigma},S_{\sigma})$ are almost isomorphic.
    \end{lemma}
    \begin{proof}
        First define:
        \begin{align*}
            s_1:&=\sum_{i=1}^k\rho(i)\\
            r_1:&=\sum_{j=1}^l\sigma(j)
        \end{align*}
        By \cite[Lemma~4.3]{ara2007} we have since $s_1=r_1$ in $\mathcal{M}(G)$ we have two finite series of elements of $\mathcal{F}(G)$: $s_1,\dots,s_n$ and $r_1,\dots,r_m$ s.t.~$(s_1,s_2),\dots,(s_{n-1},s_n)\in R_0(G)$, $(r_1,r_2),\dots,(r_{m-1},r_m)\in R_0(G)$ and $s_n=r_m$ in $\mathcal{F}(G)$. Now if we set $s_0=S_{\rho}$ we have $(s_0,s_1)\in R_0(G_{\rho})$ and thus $(s_0,s_1),\dots,(s_{n-1},s_n)\in R_0(G_{\rho})$ and by the same argument as in the proof of \lemref{lem13} we have some basis $B_{\rho}$ of a cofinite inescapable subspaces of $\mathcal{T}(G_{\rho},S_{\rho})$ s.t.
        \[s_n=\sum_{b\in B_{\rho}}T(b)\]
        Analogously we also have a basis $B_{\sigma}$ of a cofinite inescapable subspaces of $\mathcal{T}(G_{\sigma},S_{\sigma})$ s.t.
        \[r_m=\sum_{b\in B_{\sigma}}T(b)\]
        and since $s_n=r_m$ in $\mathcal{F}(G)$ we have a bijection:
        \[\pi:B_{\rho}\to B_{\sigma}\]
        with $\forall p\in B_{\rho},\ T(p)=T(\pi(p))$. This gives us an isomorphism between the cofinite inescapable subspaces with the bases $B_{\rho}$ and $B_{\sigma}$. Thus the trees $\mathcal{T}(G_{\rho},S_{\rho})$ and $\mathcal{T}(G_{\sigma},S_{\sigma})$ are almost isomorphic.
    \end{proof}
    This gives us a way to identify when two spider products with the same components are almost isomorphic to each other:
    \begin{theorem}\label{thm4}
        For any two non redundant spider products:
        \begin{align*}
            (G_{\rho},S_{\rho}):&=\llangle (G_1,\rho_1.k_1),\dots,(G_n,\rho_n.k_n) \rrangle\\
            (G_{\sigma},S_{\sigma}):&=\llangle (G_1,\sigma_1.l_1),\dots,(G_n,\sigma_n.l_n) \rrangle\\
        \end{align*}
        then $\mathcal{T}(G_{\rho},S_{\rho})$ and $\mathcal{T}(G_{\sigma},S_{\sigma})$ are almost isomorphic if and only if for each $m\in\{1,\dots,n\}$ we have:
        \[\sum_{i=1}^{k_m}\rho_m(i)=\sum_{j=1}^{l_m}\sigma_m(j)\]
        in $\mathcal{M}(G_m)$.
    \end{theorem}
    \begin{proof}
        We will first prove the theorem in the case when $n=1$. For this set $G:=G_1$.\\
        The ``if'' direction of the implication follows directly from \Lemref{lem15}.\\
        For the converse implication we can take inescapable cofinite subspaces\\
         $\mathcal{S}_{\rho}\subsetneq \mathcal{T}(G_{\rho},S_{\rho})$, $\mathcal{S}_{\sigma}\subsetneq \mathcal{T}(G_{\sigma},S_{\sigma})$, that are isomorphic. Let $B_{\rho}$ and $B_{\sigma}$, be their bases, thus we have a bijection:
        \[\pi:B_{\rho}\to B_{\sigma}\]
        s.t.~$\forall b\in B_{\rho},\ \mathcal{T}(G_{\rho},S_{\rho})_b\cong_R\mathcal{T}(G_{\sigma},S_{\sigma})_{\pi(b)}$. However since $G$ is non-redundant and each $b\in B_{\rho} $ has $T(b)\in G$ (as the subspaces are not equal to the whole tree), we must have by \lemref{lem4} $T(b)=T(\pi(b))$. thus:
        \[\sum_{b\in B_{\rho}}T(b)=\sum_{b'\in B_{\sigma}}T(b')\]
        in $\mathcal{M}(G)$, so by \lemref{lem14} we have:
        \[\sum_{i=1}^{k_1}\rho_1(i)=\sum_{b\in B_{\rho}}T(b)=\sum_{b'\in B_{\sigma}}T(b')=\sum_{j=1}^{l_1}\sigma_1(j)\]
        in $\mathcal{M}(G)$, giving us the theorem in the case $n=1$.\\
        Now in order to show the general result it suffices to show the following claim:
        \begin{clm}\label{clm3}
            $\mathcal{T}(G_{\rho},S_{\rho})$ and $\mathcal{T}(G_{\sigma},S_{\sigma})$ are almost isomorphic if and only if for each $i\in\{1,\dots,n\}$: $\mathcal{T}(\llangle(G_i,\rho_i,k_i)\rrangle)$ and $\mathcal{T}(\llangle(G_i,\sigma_i,l_i)\rrangle)$ are almost isomorphic.
        \end{clm}
        \begin{proof}
            Here we will treat $\mathcal{T}(\llangle(G_i,\rho_i,k_i)\rrangle)$ and $\mathcal{T}(\llangle(G_i,\sigma_i,l_i)\rrangle)$ as subtrees of $\mathcal{T}(G_{\rho},S_{\rho})$ and $\mathcal{T}(G_{\sigma},S_{\sigma})$ respectively that contain the same root.\\
            Thus for the isomorphic cofinite inescapable subspaces $\mathcal{S}^{\rho}\subsetneq\mathcal{T}(G_{\rho},S_{\rho})$ and $\mathcal{S}^{\sigma}\subsetneq\mathcal{T}(G_{\sigma},S_{\sigma})$, the subgraphs $\mathcal{S}^{\rho}_i:=\mathcal{S}^{\rho}\cap\mathcal{T}(\llangle(G_i,\rho_i,k_i)\rrangle)\subseteq \mathcal{T}(\llangle(G_i,\rho_i,k_i))\rrangle$ and $\mathcal{S}^{\sigma}_i:=\mathcal{S}^{\sigma}\cap\mathcal{T}(\llangle(G_i,\sigma_i,k_i)\rrangle)\subseteq \mathcal{T}(\llangle(G_i,\sigma_i,k_i)\rrangle)$ are cofinite inescapable subspaces. Now if we take an isomorphism $\Phi:\mathcal{S}^{\rho}\to\mathcal{S}^{\sigma}$ we have for any path $p\in V\mathcal{S}^{\rho}$: $(G_{\rho})_{T(p)}\cong_R(G_{\sigma})_{T(\Phi(p))}$ as the graphs are non redundant.\\
            However since $\mathcal{S}_{\sigma}$ does not contain $\epsilon_{S_{\sigma}}$ we must have $T(\Phi(p))\neq S_{\sigma}$ giving us $(G_{\sigma})_{T(\Phi(p))}=(G_{\rho})_{T(\Phi(p))}$. So since $G_{\rho}$ is non redundant we have\\
            $T(p)=T(\Phi(p))$. So taking any $p\in \mathcal{S}^{\rho}_i$ we must have $T(\Phi(p))\in G_i$ and thus $\Phi(p)\in \llangle(G_i,\sigma_i,k_i)\rrangle$. So the function $\Phi|_{\mathcal{S}^{\rho}_i}$ is an isomorphism between $\mathcal{S}^{\rho}_i$ and $\mathcal{S}^{\sigma}_i$ showing that $\mathcal{T}(\llangle(G_i,\rho_i,k_i)\rrangle)$ and $\mathcal{T}(\llangle(G_i,\sigma_i,k_i)\rrangle)$ are almost isomorphic.\\
            For the converse implication take $\mathcal{S}_i^{\rho}$ and $\mathcal{S}_i^{\sigma}$ to be cofinite inescapable subspaces of $\mathcal{T}(\llangle(G_i,\rho_i,k_i)\rrangle)$ and $\mathcal{T}(\llangle(G_i,\sigma_i,l_i)\rrangle)$ resp. s.t. $\mathcal{S}_i^{\rho}\cong\mathcal{S}_i^{\sigma}$ for each $1\leq i\leq n$. Now let $\Phi_i:\mathcal{S}_i^{\rho}\to\mathcal{S}_i^{\sigma}$ be isomorphisms between the subspaces. Now we can define:
                \begin{align*}
                    \mathcal{S}^{\rho}:&=\bigcup_{i=1}^n \mathcal{S}_i^{\rho}\\
                    \mathcal{S}^{\sigma}:&=\bigcup_{i=1}^n \mathcal{S}_i^{\sigma}
                \end{align*}
            since all the $\mathcal{S}_i^{\rho},\mathcal{S}_i^{\sigma}$ are cofinite and inescapable, the subspaces $\mathcal{S}^{\rho}$ and $\mathcal{S}^{\sigma}$ are also cofinite and inescapable. We can define the function $\Phi:\mathcal{S}^{\rho}\to\mathcal{S}^{\sigma}$ by setting $\Phi|_{\mathcal{S}_i^{\rho}}:=\Phi_i$ for each $1\leq i\leq n$. $\Phi$ is well-defined since the intersection of all the $\mathcal{S}_i^{\rho}$ consists of just the vertex $\epsilon_{S_{\rho}}$ and for each $1\leq i\leq n$ we have $\Phi_i(\epsilon_{S_{\rho}})=\epsilon_{S_{\sigma}}$. Furthermore $\Phi$ is an isomorphism since all the $\Phi_i$ are. This shows that $\mathcal{T}(G_{\rho},S_{\rho})$ and $\mathcal{T}(G_{\sigma},S_{\sigma})$ are almost isomorphic.
        \end{proof}
        Now if for each for each $m\in\{1,\dots,n\}$ we have:
        \[\sum_{i=1}^{k_m}\rho_m(i)=\sum_{j=1}^{l_m}\sigma_m(j)\]
        than as we have already shown, each $\mathcal{T}(\llangle(G_m,\rho_m,k_m)\rrangle)$ and $\mathcal{T}(\llangle(G_m,\sigma_m,l_m)\rrangle)$ are almost isomorphic we get $\mathcal{T}(G_{\rho},S_{\rho})$ and $\mathcal{T}(G_{\sigma},S_{\sigma})$ are almost isomorphic.\\
        Conversely if $\mathcal{T}(G_{\rho},S_{\rho})$ and $\mathcal{T}(G_{\sigma},S_{\sigma})$ are almost isomorphic, we must have $\mathcal{T}(\llangle(G_m,\rho_m,k_m)\rrangle)$ and $\mathcal{T}(\llangle(G_m,\sigma_m,l_m)\rrangle)$ be almost isomorphic. thus as we have already shown for each $m\in\{1,\dots,n\}$:
            \[\sum_{i=1}^{k_m}\rho_m(i)=\sum_{j=1}^{l_m}\sigma_m(j)\]
    \end{proof}
    Since by  \Lemref{lem7} each label-regular directed tree is almost isomorphic to a non-redundant spider product with components that are robust and connected. By \Lemref{lem11} these components are unique up to isomorphism. Finally by \Thmref{thm4} we can change the vertices the root is connected to as long as their sums remain the same in the graph monoid. This allows us to associate each label-regular directed tree with an unique finite collection of connected robust non-redundant graphs s.t. their union is also non-redundant, together with a collection of members of their graph monoid that can be written as a sum of vertices s.t. their half graphs cover the whole graph.\\   
    Now we can see that in order to determine when two rooted graphs have almost isomorphic unfolding trees, we have to find via \lemref{lem14} a spider product of connected robust graphs that has almost isomorphic unfolding trees to the unfolding trees of the original graphs. Now we can reduce each of the spider products to being non-redundant spider products. Now if we don't have a one-to-one correspondence between isomorphic components of the spider products, the unfolding trees cannot be almost isomorphic. If such a correspondence exists we can compare the sums of the neighbors in each components (in the monoid of the component) as in \thmref{thm4}.\\
    This outlines an algorithm that determines when two graphs have almost isomorphic unfolding trees. As this algorithm involves solving the word problem in some commutative finitely generated monoid and this is a problem  with exponential complexity \cite{MAYR1982}, this might be an indication that checking whether two graphs have almost isomorphic unfolding trees is not solvable in polynomial time .
    \section{Example}
    As an example we can focus on the graphs that have only 2 vertices. \\
    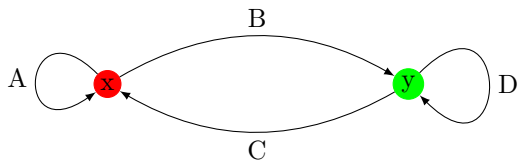
\begin{figure}[ht]
        \begin{center}
            \begin{tikzpicture}
                \node[red node]     (x) at (-2,0)   {x};
                \node[green node]   (y) at (2,0)    {y};

                \path[-latex]
                        (x)     edge[bend left]                         node[above]             {B} (y)
                                edge[out=135, in=215, looseness=15]     node[left]              {A} (x)
                        (y)     edge[bend left]                         node[below]             {C} (x)
                               edge[out=45, in=315, looseness=15]       node[right]             {D} (y)
                        ;    
            \end{tikzpicture}
        \end{center}
    \caption{A graph with two vertices with $A,B,C,D$ denoting the multiplicity of the edges}
    \end{figure}
    We will thus take the graph with the adjacency matrix  $\begin{pmatrix}
                                                                A & B\\
                                                                C & D\\    
                                                            \end{pmatrix}$  we will denote the  vertices of this graph by $x,y$.\\
    Since any non-trivial non-edge collapsing equivalence relation $\sim$ on this graph will have to satisfy $x\sim y$. Thus we must have $A+B=|E_o(x)|=|E_o(y)|=C+D$.\\
    Conversely if we in fact have $A+B=C+D=:N$, we will denote:
     \[E_o(x)=:\{e_1,\dots,e_N\} \text{ and } E_o(y)=:\{d_1,\dots,d_N\}\]
    We can define the equivalence relation $\sim$ by setting $x\sim y$ and:
    \[\forall i\in \{1,\dots,N\},\, e_i\sim d_i \]
    It is clear that this is in fact a non-edge collapsing relation. When we factor over this equivalence relation we get a graph  with only one vertex with $N$ loops, that is of course non-redundant. So two graphs $G,H$ with $2$ vertices each have isomorphic unfolding trees based on $x$ iff they have they have the adjacency matrices $\begin{pmatrix}
                            A & B\\
                            C & D\\    
                        \end{pmatrix}$ and $\begin{pmatrix}
                                                A' & B'\\
                                                C' & D'\\    
                                            \end{pmatrix}$ resp., s.t.
        \[A+B=C+D=A'+B'=C'+D'\]
        or $A=A'$, $B=B'$, $C=C'$ and $D=D'$.\\
    Now to check when two unfolding trees of graphs with two vertices are almost isomorphism we will assume that they are non-redundant i.e. $A+B\neq C+D$. Then if the unfolding tree based on any root of a graph of two vertices is isomorphic to another unfolding tree of a graph of two vertices, the graphs have to be isomorphic. Thus we only have to determine whether the unfolding tree based on one root is isomorphic to the one based on the other. Additionally in order $x$ and $y$ to both be roots we can assume that $B,C>0$. We will have to determine when the unfolding tree based on $x$ is almost isomorphic to the one based on $y$, by looking at when $x=y$ in the graph monoid. To do this we will apply the algorithm CPC from \cite[Chapter 6]{rosales1999} which allows us to generate a canonical presentation of a monoid that allows us to solve the word problem. We will furthermore assume that $A\geq D$ by swapping $x$ and $y$. \\
    First we will look at what happens when $D=0$, in this case the graph looks as follows:\\
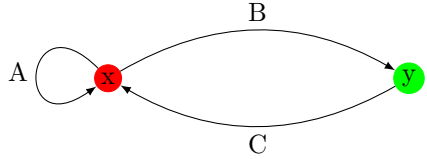
        \begin{center}
            \begin{tikzpicture}
                \node[red node]     (x) at (-2,0)   {x};
                \node[green node]   (y) at (2,0)    {y};
                \path[-latex]
                (x)     edge[bend left]                         node[above]             {B} (y)
                        edge[out=135, in=215, looseness=15]     node[left]              {A} (x)
                (y)     edge[bend left]                         node[below]             {C} (x)
                ;    
            \end{tikzpicture}
        \end{center}
    The graph monoid is presented by the equations:
        \[  \begin{cases}
                x=Ax+By\\
                y=Cx
            \end{cases}\]
    these are equivalent to:
        \[  \begin{cases}
                x=(A+BC)x\\
                y=Cx
            \end{cases}\]
    Thus the monoid consist only of multiples of $x$ and we have $Mx=Nx$ iff $M,N\neq 0$ and $M\equiv N \mod A+CB$ or $M=N=0$. From here we can see that since $y=Cx$ in order to have $x=y$ we must have $C=1$, or $B=1$ and $A=0$. Which are thus the only cases when the unfolding trees based on $x$ and $y$ are almost isomorphic to each other.\\
    Now we will be able to assume that $A,D>0$, and by swapping $x$ and $y$ we may assume $A\geq D$.\\
    Our main family of examples are the ones where $B=1$ and $D=2$. Then the graph monoid is presented by:
    \begin{equation}\label{eq3}
        \begin{cases}
            x=Ax+y\\
            y=Cx+2y
        \end{cases}
    \end{equation}
     We can write $A=C+N+1$ for $N=A-C-1$. There is some $k$ and $0\leq r<N$ s.t. $kN=A+r$. Now repeating the calculations we did before we get:
    \[x=Cx+(A+r)x+y+x=(C+2+r)x\]
    so:
    \[x=(C+1+N)x+y=x+(N-r-1)x+y=(N-r)x+y\]
    and:
    \[(r+2)x+y=(C+r+2)x+2y=x+2y\]
    if we plug this into the second equation of \ref{eq3} we get:
    \begin{multline*}
        y=Cx+2y=(C-1)x+x+2y=(C-1)x+(r+2)x+y=(C+r+N)x+2y=\\
        (C+r+N-1)x+(x+2y)=(C+r+N-1)x+(2x+y)=(C+r+N+1)x+y=(r+1)x
    \end{multline*}
    we also have:
    \[x=(N-r)x+y=(N+1)x\]
    Now to show that $x\neq y$ we will have to show that the conditions of \ref{eq3} are implied by and thus equivalent to:
    \begin{equation}\label{eq34}
        \begin{cases}
            x=(N+1)x\\
            y=(r+1)x
        \end{cases}
    \end{equation}
    This can be shown by iterating:
    \[Ax+y=(A+r+1)x=(kN+1)x=((k-1)N+1)x=\dots=(0N+1)x=x\]
    and:
    \begin{multline*}
        Cx+2y=(C+2r+2)x=((C+1)+2r+1)x=(C+1+r+r+1)x=\\
        ((k-1)N+r+1)=((k-2)N+r+1)x=\dots=(r+1)x=y
    \end{multline*}
    So in order to have $Kx=Lx$ (for $K,L>0$) in the graph monoid we have to have $K\equiv L\mod N$. So we have $x=y$ if and only if $r+1\equiv 1\mod N$, and since $0\leq r<N$ we must have $r=0$. So the unfolding trees based on $x$ and on $y$ are almost isomorphic if and only if $N|A$, i.e. $A-C-1|A$.\\
    Now note that if we have a graph $G$ with two vertices and the adjacency matrix:            
    $\begin{pmatrix}
        A & B\\
        C & D\\    
    \end{pmatrix}$,
    with $A-C\geq0$ and $B-D\geq0$ then the graph with the adjacency matrix:
    $\begin{pmatrix}
        A-C & B-D+1\\
        C & D\\    
    \end{pmatrix}$
    have the same graph monoid to $G$, since:
        \[  \begin{cases}
                x=Ax+By\\
                y=Cx+Dy
            \end{cases}\]
    is equivalent to:
        \[  \begin{cases}
                x=(A-C)x+(B-D+1)y\\
                y=Cx+Dy
            \end{cases}\]
    Analogously if we have $C-A\geq0$ and $D-B\geq0$, the graph with the adjacency matrix:
    $\begin{pmatrix}
        A & B\\
        C-A+1 & D-B\\    
    \end{pmatrix}$, also has the same graph monoid as $G$. As long as $B,C>0$ these reductions will decrease $A+D$. So we can apply these reductions repeatedly to get to a graph with either $A=0$,$D=0$, or $(A-C)(B-D)<0$. The case when $A=0$ or $D=0$ was tackled before. Now if $(A-C)(B-D)<0$ the question becomes more difficult. I have checked all cases with $A,B,C,D\leq 100$ and there for every such graph with $A\geq D$ and $x=y$ in the graph monoid we have $B=1$ and $D=2$. This leads me to the conjecture:
    \begin{conj}
        For any graph $G$ on the vertex set $\{x,y\}$ and the adjacency matrix
       $\begin{pmatrix}
            A & B\\
            C & D\\    
        \end{pmatrix}$, with $A\geq D$, $A,B,C,D>0$ and $(A-C)(B-D)<0$, we have $x=y$ in $\mathcal{M}(G)$ if and only if $B=1$, $D=2$ and $(A-C-1)|A$.
    \end{conj}
    \section{Acknowledgments} 
    I’d like to thank my PhD supervisors: George Willis, Colin Reid and Stephan Tornier for advising me during the writing of this paper and proofreading it. I’d also like to thank Roozbeh Hazrat for introducing me to the graph monoid, during a "Symmetry in Newcastle" seminar. The calculations for the last chapter, were made in GAP \cite{GAP4} and C++. The writing of this paper was supported by the grant FL170100032 form the Australian Research Council.

    \bibliographystyle{plain} 
    \bibliography{bibliography}
\end{document}

%% file: node.tikzstyles
\tikzstyle{dot}=[fill=black, draw=black, shape=circle]
\tikzstyle{red_dot}=[fill=red, draw=red, shape=circle]
\tikzstyle{green_dot}=[fill=green, draw=green, shape=circle]
\tikzstyle{yellow_dot}=[fill=yellow, draw=yellow, shape=circle]
\tikzstyle{blue_dot}=[fill=blue, draw=blue, shape=circle]
\tikzstyle{orange_dot}=[fill=orange, draw=orange, shape=circle]

\tikzstyle{arrow}=[->]